\documentclass[12pt,a4paper]{article}
\usepackage{amssymb,enumerate,verbatim}
\usepackage{amsgen,amsmath,amsthm,amstext,amsbsy,amsopn,amsfonts,amssymb}
\usepackage{cite}
\usepackage[all,cmtip]{xy}

\newtheorem{Def}{Definition}
\newtheorem{Thm}{Theorem}
\newtheorem*{Thm*}{Theorem}
\newtheorem{Prop}{Proposition}
\newtheorem*{Prop*}{Proposition}
\newtheorem{Lem}{Lemma}
\newtheorem{Cor}{Corollary}

\newtheorem{Rem}{Remark}
\newcommand\Z{\mathbb{Z}}
\newcommand\C{\mathbb{C}}

\newcommand\R{\mathbb{R}}

\newcommand\F{\mathbb{F}}
\newcommand\N{\mathbb{N}}

\newcommand\oo{\omega}
\newcommand\VV{\mathcal{V}}
\newcommand\NN{\mathcal{N}}
\newcommand\WW{\mathcal{W}}
\newcommand\TT{\mathcal{T}}

\newcommand\DD{\mathcal{D}}

\newcommand\XX{\mathcal{X}}
\newcommand\SSS{\mathcal{S}}
\newcommand\YY{\mathcal{Y}}

\newcommand\FF{\mathcal{F}}
\newcommand\UU{\mathcal{U}}
\newcommand\BB{\mathcal{B}}
\newcommand\HH{\mathcal{H}}

\newcommand\e{\varepsilon}

\title{On non-amenable embeddable spaces in relation with free products}
\author{K\'evin BOUCHER}

\begin{document}
\maketitle

\begin{abstract}
In this paper we give sufficient conditions for a free product of residually finite groups to admit an embeddable box space. This generalizes the constructions of Arzhantseva, Guentner and Spakula in \cite{AGS} and gives a new class of non-amenable metric spaces with bounded geometry which coarsely embeds into Hilbert space.\\
\end{abstract}

\section{Introduction}

The concept of coarse embedding of a metric space into another was inspired by Gromov's work \cite{G} in relation with the Novikov conjecture on the homotopy invariance of higher signatures for closed manifolds :
\begin{Def}\label{def:def1}
Let $(X,d_X)$ and $(Y,d_Y)$ be two metric spaces. We say that $X$ coarsely embeds into $Y$ if, there exist a map $F:X\rightarrow Y$ and two proper functions $\rho_\pm:\mathbb{R}_+\rightarrow\mathbb{R}_+$ such that:
\[\rho_-(d_X(x,y))\le d_Y(F(x),F(x'))\le\rho_+(d_X(x,x'))\]
for all $x, x'\in X$.
\end{Def}
This tells us that two points are far from each other in $X$ if and only if they are far in $F(X)$.

This geometric notion was related by Yu \cite{Yu} to the coarse Baum-Connes conjecture and in the case of Cayley graphs of groups  to the Novikov conjecture.\\ 
In that remarkable paper the author introduced a notion of amenability for metric spaces, called property (A), which implies coarse embedding into Hilbert spaces:
\begin{Def}\label{def:def2}
A discrete metric space $(X,d)$ with bounded geometry has property (A) if and only if for every $R>0$ and $\varepsilon>0$, there exist $S>0$ and a function $\phi:X\rightarrow \ell^2(X)$ with $\|\phi(x)\|=1$ for any $x\in X$ such that:
\begin{description}
\item{$(1)$}  for any $x_1,x_2\in X$ with $d(x_1,x_2)\le R$, $|1-\langle\phi(x_1),\phi(x_2)\rangle|\le\varepsilon$, and
\item{$(2)$}  for any $x\in X$, $\{\phi(x)\neq0\}\subset B_S(x)$.
\end{description}
\end{Def}
Note that this definition of property (A) is not the original but an equivalent one for metric spaces with bounded geometry.\\
Definition \ref{def:def2} can be compared to the following notion of amenability for discrete countable groups:
\begin{Def}\label{def:def3}
A discrete group $G$ with a proper length function $\ell$ is amenable if and only if for every $R>0$ and $\varepsilon>0$, there exists a function $\xi\in\ell^2(G)$ with $\|\xi\|=1$ such that  for all $g\in B(e,R)$, $\|\lambda(g)\xi-\xi\|\le\varepsilon$ where $\lambda: G\rightarrow \ell^2(G)$ is the left regular representation.
\end{Def}
Indeed, for given $\e>0$ and $R>0$, the map $\phi(g)=\lambda(g)\xi$ verifies the conditions of Definition \ref{def:def2} and so property (A) can be thought as a non-equivariant notion of amenability.\\
Actually property (A) can be characterized by the amenability of a certain groupo\"id associated to the given metric space, called translation groupo\"id (cf. \cite{STY}).

Numbers of coarse properties imply property (A) and make it, in some situations, easier to check than coarse embeddability.
For example: any metric space of finite asymptotic dimension \cite{DG}, as hyperbolic spaces \cite{Roe2}, has property (A).

In the same way, coarse embeddability of a metric space into a Hilbert space can also be identified to a weaker notion of amenability, called a-T-menability.
Indeed, a metric space coarsely embeds into Hilbert space if and only if its translation groupo\"id is a-T-menable \cite{STY}.
Let us recall the definition of a-T-menability in the group setting for the sake of simplicity:
\begin{Def}\label{def:def4}
A discrete countable group $G$ is a-T-menable if it acts properly on a Hilbert space by isometries.
\end{Def}

In group theory it is well known that the free group of rank 2 is an example of finitely generated non-amenable but a-T-menable group. 
In the case of metric spaces with bounded geometry, find a metric space without (A) but still embeddable into a Hilbert space was a crucial question answered in \cite{AGS}. 
In that paper the authors exhibited a \textbf{box space} (cf. below) associated to the rank 2 free group which embeds and use a girth criterion to justify that it does not have property (A) \cite{WY}.\\

Before we introduce box spaces we must talk about coarse disjoint union of metric spaces.\\
Let $\N$ be the set of integers that we endow with a \textit{far away} distance $d_{\text{f.a}}$ given by the formula $d_{\text{f.a}}(n,m)=|2^n-2^m|$.
The important aspect of $d_{\text{f.a}}$ is that in the neighborhoods of the infinity the points of $\N$ are arbitrarily far away from each other.
Indeed $d_{\text{f.a}}(n,m)\ge2^{\max(n,m)-1}$ for all $n,m\in\N$ with $n\neq m$ and so $d_{\text{f.a}}(n,m)\to\infty$ when $n+m\to\infty$ with $n\neq m$.

\begin{Def}\label{def:coarse}
Let $((X_n, d_n))_n$ be a sequence of metric spaces.\\
The coarse disjoint union $(\sqcup_n X_n, d)$ is the unique metric space $(X,d)$ up to coarse equivalence such that:
\begin{description}
\item{$(1)$}  there exist a map $p:X\to\N$ and a proper function $\rho:\R_+\to\R_+$ such that $p^{-1}(n)$ is isometric to $X_n$ for all $n$  and $d_{\text{f.a}}(p(x),p(y))\le \rho(d(x,y))$ for all $x,y\in X$.
\item{$(2)$}  For all $R>0$, $U_{n,m}=\{x\in X_n\mid d(x,X_m)\le R\}$ is a bounded set of $X_n$ for all $n,m$.
\end{description} 
\end{Def}

As it is explained in Section \ref{sec:sec3} there is a correspondence between sequence of metric spaces with their arrows (cf. Definition \ref{def:def8}) and coarse disjoint union $(X,d,p,\rho)$ with classical coarse arrows that commute with the map $p$.\\

Let $G$ be a discrete residually finite group endowed with a proper length $\ell$ and $\NN=(N_n)_n$ a sequence of finite index normal subgroups with trivial intersection.
\begin{Def}\label{def:def5}
The box space of $G$ with respect to $\NN$, denoted by $\Box_\NN G$, is the coarse disjoint union $\sqcup_n G/{N_n}$, where each finite quotient $Q/{N_n}$ is endowed with the quotient length induced by $\ell$ (cf. Section \ref{sec:sec3}).
\end{Def}
Box spaces are defined up to coarse equivalence and by Proposition \ref{prop:prop4} and Corollary \ref{cor:cor1} they do not depend on the choice of $\ell$.\\

One of the key points in box space theory is that geometric properties on a group box space translate to analytic/harmonic properties on this group:
\begin{align}
G \text{ is amenable } &\Longleftrightarrow \Box_{\NN}G \text{ has property A,}\\
G \text{ is a-T-menable } &\Longleftarrow \Box_{\NN}G \text{ is coarsely embeddable into a Hilbert space,}\\
G \text{ has property (T) } &\Longrightarrow \Box_{\NN}G \text{ is an expander.}
\end{align}
The Proofs of the first two facts can be found in \cite{Roe} and the third in \cite{Mar}.\\
Recall that property (T) is a strong negation of a-T-menability in the following sense:
\begin{Def}\label{def:def6}
A discret countable group $G$ has property (T) if any isometric action of $G$ on a Hilbert space has a fixed point.
\end{Def}
Analogously to the relation between property (T) and a-T-menability, the notion of expansion, mentioned in the correspondance $(3)$ and explained in Section \ref{sec:sec5}, is an obstruction to coarse embeddability of a given metric space into Hilbert space (cf. \cite{Lu} and \cite{Mar} for more about expansion).\\
Let us mention that in some sense the converse of the correspondence $(2)$ exists.\\
A generalized notion of coarse embedding called fibered coarse embedding was introduced in \cite{CWW}.
Since certain expander sequences can admit such kind of embedding, it appears to be strictly weaker than the original notion of coarse embedding.\\
The strength of this generalization is that it can be associated to any geometric action of a residually finite group on a metric space. 
Moreover, in the context of $L^p$ spaces, there is an equivalence between the existence of a proper action on a $L^p$ space and existence of a fibered coarse embedding of its box space into it\cite{Arn}.\\

In this paper we investigate finite quotients of free products.\\
Let's consider $A$ and $B$ two discrete residually finite groups that we suppose to be finitely generated.\\
Let $G=A\star B$ be their free product (which is residually finite cf. \cite{BT}) and $\ell$ a proper length on $G$.\\
The main result of this paper is:
\begin{Thm}\label{thm:thm1}
If $A$ and $B$ are amenable, then $G$ admits a box space which coarsely embeds into Hilbert space.
\end{Thm}
In other words, as we use this terminology in the rest, $G$ has a faithful sequence of finite quotients which coarsely embeds into Hilbert space (cf. Lemma \ref{lem:lem1}).
\begin{Def}
Let $G$ be a residually finite group.
A infinite sequence of (right) quotient $(Q_n)_n$ of $G$ is called faithful if any element of $G$ acts non-trivially by left translation on all but a finite number of these quotients.
\end{Def}
Since $G$ is non-amenable when $|A|,|B|\ge3$, Theorem \ref{thm:thm1} and the box space correspondence $(1)$ described before gives us a new class of non-amenable metric spaces with bounded geometry which coarsely embeds into Hilbert space.\\
Moreover, correspondence $(2)$ combined with our result give another proof with a different flavor of the well known result: 
\begin{Thm*}
Let $A$ and $B$ be two discrete countable amenable groups.
Then their free product $G=A\star B$ is a-T-menable \cite{CCJJV}.
\end{Thm*}

We are also interested by what can be done beyond the free product of amenable groups case.\\
Given a free product $G=A\star B$, a necessary condition for it to admit an embeddable box space is that its factors $A$ and $B$ admit such box spaces (cf Proposition \ref{prop:prop6}). 
So a natural question is: does the free product of discrete residually finite groups which have an embeddable box space have itself such an embeddable box space?\\
It is not clear if our method can be adapted to treat this situation.
As it is explained in the Outline, a part of our approach uses an \textit{amenable extension} type result adapted to spaces obtained as coarse disjoint union.
But even in the context of classical metric spaces we cannot expect, in full generality, that an \textit{extension of embeddable spaces} to be itself embeddable (cf. \cite{AzTe}).\\
However we have the following partial answer:\\
Let $A$, $B$ and $G$ as above. 
We assume that $A$ and $B$ are both finitely generated and we fix a finite generating set $S\subset G$ of $G$. 
\begin{Prop}\label{prop:prop1}
If $A$ and $B$ have an embeddable box space then  $G$ has a faithful sequence of irreducible symmetric representations $(\sigma_n: G\rightarrow Sym(X_n))_n$ such that the coarse disjoint union of Schreier graphs $\sqcup_nSch_S(X_n)$ coarsely embeds into Hilbert space.
\end{Prop}
As explained in Section \ref{sec:sec5} such symmetric representations are nothing but quotients induced by non-necessary normal subgroups.\\
Even if sequences of symmetric representations still preserve the direct arrow of correspondence $(1)$ and the implication $(3)$.
We have to keep in Definition \ref{def:def5} of a box space the assumption that the sequence of subgroups we consider are normal to preserve some part of the correspondences described before.
Indeed the following result, proven in Section \ref{sec:sec5}, shows that considering symmetric representations instead of finite quotients breaks the correspondence $(1)$:
\begin{Prop}\label{prop:prop2}
Let $\F_2$ be the free group of rank 2. There exists a faithful sequence of symmetric representations $(\sigma_n: \F_2\rightarrow Sym(X_n))_n$ of $\F_2$ such that the sequence of Schreier graphs $\sqcup_nSch(X_n)$ has property (A).
\end{Prop}

Our investigation on free products was motived by the following observation:\\
Consider a free product $G=A\star B$. $G$ acts without bounded orbits on a tree (cf. \cite{Ser} and Section \ref{sec:sec4}) and so does not have property (T) \cite{BHV}. This finds an analogue \textit{at finite scale} in box space theory:
\begin{Prop}\label{prop:prop3}
Let $A$, $B$ be two residually finite groups and $G=A\star B$ their free product. Then $G$ has a non-expander box space.
\end{Prop}
In other words, we cannot expect a free product with property $(\tau)$ \cite{Lu} even if both of its factors has property (T).
The proof, given in Section \ref{sec:sec5}, uses only elementary considerations but as far as we know has never been written. 

We finish this introduction by a question. It deserves to be noted that all known examples of non-amenable embeddable metric spaces come from groups which admit  (or contain a large subgroup who does) an action without bounded orbits on a tree.\\ 
Can we find an example of embeddable non-amenable box space coming from a group with property (FA) \cite{BHV}?\\
Recall that a group has property (FA) if all of its actions on a tree have a fixed point.

\section{Outline}

Here we explain the idea behind the proof of Theorem \ref{thm:thm1} and detail the structure of this paper.\\

Let us start with the classic case of discrete metric spaces.\\
We consider the free product of two amenable groups $G=A\star B$ endowed with a proper length function $\ell$.\\
The group $G$ can be described as an extension of the direct product $A\times B$ by a free subgroup $D$ (cf. Lemma \ref{lem:d}):
$$1\rightarrow D\rightarrow G\rightarrow A\times B\rightarrow 1$$
At first look this exact sequence is only algebraic and apriori does not give any information on the geometry of $(G,\ell)$.
Now if we endow $D$ with the restricted length $\ell|_D$ and $A\times B$ with the quotient length $\overline{\ell}$ induced by $\ell$.
Some aspects of the coarse geometry of $(G,\ell)$ as uniform embeddability can be understood through the geometry of $(D,\ell|_D)$ and $(A\times B,\overline{\ell})$.\\
Indeed, if $A$, $B$ and so $A\times B$ are amenable, they have property (A) as metric groups with respect to any proper length and by a theorem of \cite{DMG} the coarse embedding of $(G,\ell)$, can be reduced to the coarse embedding of $(D,\ell|_D)$. 
Moreover since proper length functions on a given group are unique up to coarse equivalence.
The length function $\ell|_D$ on $D$ can be replaced by any other proper length function.
So a trick to prove the uniform embedding of $D\simeq\F_\infty$ endowed with $\ell|_D$ is to realize it as a subgroup of the rank 2 free group $\F_2$ endowed with a word metric on a basis for which it is known that an uniform embedding exists (cf. \cite{BrOz}) and then use the uniqueness up to coarse equivalence of proper lengths on $\F_\infty$ to conclude.\\
Note that as soon as $G$ is non-trivial and $A$ or $B$ has infinite cardinality, $D$ has infinite rank (cf. Proposition \ref{prop:prop7bis}) and so the class of proper length on $D$ are not the natural one. 
In particular $\ell|_D$ cannot be replaced by a word length on a generating of $D$, since the latter distance is not proper on the infinite rank free group.\\

Our approach, under the assumption that $A$ and $B$ are residually finite, is to do something similar at finite scale.\\
Let $\overline{A}$ and $\overline{B}$ be two finite quotients of respectively $A$ and $B$.
We denote $\overline{G}=\overline{A}\star\overline{B}$ the free product of these two groups.\\
Then the following commutative diagram holds:
$$\xymatrix{
1\ar[r]&
D\ar[r]\ar[d]&
G\ar[r]\ar[d]&
A\times B\ar[r]\ar[d]&
1\\
1\ar[r]&
D'\ar[r]\ar[d]&
\overline{G}\ar[r]\ar[d]&
\overline{A}\times \overline{B}\ar[r]\ar[d]&
1\\
&1&1&1&\\
}$$
where the free subgroup $D'$ has, this time, finite rank (cf. Lemma \ref{lem:d}).\\
Our prototype of finite quotients has the form $\overline{G}/H$ where $H\unlhd D'$ is a finite index normal subgroup of $D'$.
In particular $\overline{G}/H$ appears as an extension of $\overline{A}\times\overline{B}$ by a large girth subgroup $D'/H$.\\
This process, which is detailed in Section \ref{sec:sec4}, gives faithful sequence under reasonable assumptions on the choice of the $\overline{A}$, $\overline{B}$ and $H$.

Now, let $(G_n)_n$ be a faithful sequence of finite quotients of $G$ built as described before.
We endow each $G_n$ with the quotient length induced by $\ell$ that we denote the same way.
Here we assume that $G_n$ is an extension of a direct product $A_n\times B_n$ of finite quotients of $A$ and $B$ by a subgroup of large girth $D_n$. \\
As in the classic case, our approach is to reduce the coarse geometry of the metric space $\sqcup_n(G_n,\ell)$ to the geometry of $\sqcup_n(D_n,\ell|_{D_n})$ and $\sqcup_n(A_n\times B_n,\overline{\ell})$.\\
The first step is to reduce the proof of Theorem \ref{thm:thm1} to the uniform embedding of the sequence $\sqcup_n (D_n,\ell|_{D_n})$ (cf. Theorem \ref{thm:thm3}) under the assumption that $A$ and $B$ are amenable.
This is done in Subsection \ref{sub:sub41} where we develop gluing techniques inspired by \cite{DG} for metric spaces obtained as coarse disjoint union.

The second step of the proof which concerns the uniform embedding of the space $\sqcup_n(D_n,\ell|_{D_n})$ is a bit more delicate.
Indeed if we refer to the preceding notations, we have to do a suitable choice of the $H$'s inside of the $D'$'s to ensure the uniform embedding of $\sqcup_n(D_n,\ell|_{D_n})$.
The difficulty here is that the proper lengths on the $D_n$'s do not correspond to the natural class of length we would like to consider on them as soon as $G$ is non-trivial and $A$ or $B$ has infinite cardinality, if we refer to our remark on the classic case.
To solve this part we compare the sequence $((D_n,\ell|_{D_n}))_n$ to the one obtained by taking of each $D_n$ the quotient length induced by $D$ endowed with a word length on a suitable basis (cf. Proposition \ref{prop:prop7bis}).
We call this new sequence the \textit{combinatorial sequence}.
Note that, this procedure is not free.
Indeed, the \textit{compression} function $\rho_{n,-}$ of $(D_n,\ell|_{D_n})\rightarrow (D_n,\ell_\text{Comb})$ (cf. Section 3) can be estimated at every stage $n$ and verifies $|A_n|+|B_n|\le \rho_{n,-}(1)$.
This inequality tells us, as suspected, that these two sequences of metric spaces are non-equivalent to each other if $G$ is non-trivial and $A$ or $B$ has infinite cardinality.
In spite of this, some \textit{local comparison} is possible and this is enough to conclude.\\
Let us explain it in the classic case:\\
Let $\F(S)$ be the free on an infinite countable set $S$ endowed with a proper length $\ell$ and $F_k=\F(S_k)$ where $S_k\subset S$ is a finite subset of $S$ with $S_k\subset S_{k+1}$ and $S=\cup_kS_k$.
Since $\ell$ is proper, for all $R\ge0$, there exists $K_0$ such that $B_\ell(e,R)\subset F_k$ for all $k\ge K_0$.
Since $F_k$ embeds uniformly a direct limit argument can be applied to prove that $\F(S)$ uniformly embeds too (cf. \cite{DMG}).

An investigation of the geometry of $D$ allows us to build a sequence  $(X_k)_k$ inside of $\sqcup_n D_n$ which plays a role similar to $F_k$ in the classic case.
The fundamental point is that each $X_k$ can be compared, with controlled compression and dilatation function, to its combinatorial version on all fibers of $\sqcup_n D_n$ except a finite number of them, in other words the two metric structures on $X_k$ match at infinity (cf. Corollary \ref{cor:cor5}).
As a consequence it appears that the uniform embedding of $\sqcup_n(D_n,\ell|_{D_n})$ is implied by the uniform embedding of its combinatorial version.
This is explained in Subsection \ref{sub:sub42} where we introduce a direct limit formalism and establish a stability result for coarse disjoint unions of groups (cf. Proposition \ref{prop:prop8}).
So to conclude Theorem \ref{thm:thm1} proof it is enough to treat the combinatorial case.
But the uniform embedding problem in this case can be solved by a slight modification of \cite{AGS} arguments (cf. Proposition \ref{prop:prop10}).

This paper is organized as follows:
we start with Section \ref{sec:sec3} where we introduce general results and terminologies that we need in the rest.
Section \ref{sec:sec4} is devoted to the proof of Theorem \ref{thm:thm1}.
We finish with Section \ref{sec:sec5} where we prove the Proposition \ref{prop:prop1}, \ref{prop:prop2} and \ref{prop:prop3}.


\section{Basics}\label{sec:sec3}
In this section we expose generalities on metric groups and spaces.\\

First, let us briefly recall terminologies and constructions about length functions.\\
Here $G$ is a discrete countable group.
\begin{Def}\label{def:def7}
A length function $\ell:G\rightarrow\mathbb{R}_+$ on $G$ is a function which satisfies:
\begin{description}
\item{$(1)$} for any $g,h\in G$, $\ell(g.h)\le\ell(g)+\ell(h)$.
\item{$(2)$} for any $g\in G$, $\ell(g)=\ell(g^{-1})$.
\item{$(3)$} $\ell(g)=0$ if and only if $g=e$.
\end{description}
We say that $\ell$ is proper if for any $R>0$, $\ell^{-1}([0,R])\subset G$ is a finite set of $G$, in other words, each ball of $G$ is finite. 
\end{Def}

Given a length $\ell$ on $G$, we associate a left-invariant metric $d$ by the formula $d(g,h)=\ell(g^{-1}h)$.
Moreover, starting with the left-invariant metric $d$, $\ell(g)=d(g,e)$ is a length function on $G$.
In the rest we pass to one to the other without distinction.\\

Let $H\unlhd G$ be a normal subgroup of $G$ and $\ell$ a proper length function on $G$.\\ 
The quotient length on $Q=G/H$ induced by $\ell$ is defined by the formula: $\ell_Q(\overline{g})=\inf_{h\in H}\ell(gh)$ where $\overline{g}\in Q$ is the class of an element $g\in G$.\\
Since $H$ is normal it is routine to prove that conditions $(1)$ and $(2)$ of Definition \ref{def:def7} are verified.\\
Note that there is no natural way to build a distance on a $G$-homogeneous space in full generality.\\
Since $\ell$ is proper, for all $\e>0$, the set of $h\in H$ such that $\ell(gh)\le\ell_Q(\overline{g})+\e$ is finite.
This ensures that we can find $h_0$ such that $\ell_Q(\overline{g})=\ell(gh_0)$ and so for any $\overline{g}\in Q$, $\ell_Q(\overline{g})=0$ implies that $g\in H$. 
Condition $(3)$ is verified.
Another consequence is that for any $R\ge0$, $\pi(B_G(g,R))=B_Q(\overline{g},R)$ where $\pi$ is the quotient map. 
This set equality ensures that $\ell_Q$ is proper.\\
 
Given a symmetric generating set $S\subset G$, the word length on $G$, $\ell:G\rightarrow\R_+$, associated to $S$ is defined by:
$$\ell(a)=\inf\{n\mid a=a_1\dots a_n,\text{with $a_i\in S$}\}$$
for $a\neq0$ and $\ell(e)=0$.
It appears that when $S$ is finite, the length $\ell$ is proper.\\
In the case of free groups we assume this generating set to be a basis \cite{LS}.

Note that, concerning the coarse geometry of a locally finite metric group the choice of the proper length is not matter as the following proposition shows.
\begin{Prop}\label{prop:prop4} 
Let $G$ be a discrete countable group.\\
Suppose $\ell_1$ and $\ell_2$ are two proper length functions on $G$. Then $(G,\ell_1)$ is coarsely equivalent to $(G,\ell_2)$.
\end{Prop}
\begin{proof}
Let's define $\rho_-(t)=\inf_{\{g\mid\ell_1(g)\ge t\}}\ell_2(g)$ and $\rho_+(t)=\sup_{\{g\mid\ell_1(g)\le t\}}\ell_2(g)$ which are well defined because of the properness assumption on $\ell_1$. 
These two increasing functions are called respectively \textbf{compression} and \textbf{dilatation}. 
They verify the following inequalities:
$$\rho_-(d_1(g,h))\le d_2(g,h)\le\rho_+(d_1(g,h))$$
for any $g,h\in G$.
Moreover, they are respectively the largest and smalest control function such that the above inequalities hold.
Actually, coarse embeddability of a space into another is equivalent to the existence and properness of $\rho_\pm$.\\
Using their definitions, in order to verify that $\rho_-$ (respectively, $\rho_+$) is proper, it is enough to prove that, for all $R>0$, there exists $R'>0$ such that $\ell_2(g)\le R$ (respectively, $\ell_1(g)\le R$) implies $\ell_1(g)\le R'$ (respectively, $\ell_2(g)\le R$).
But since $\ell_2$ (respectively, $\ell_1$) is proper and $G$ is discrete, the set $\{g\mid\ell_2(g)\le R\}$ (respectively, $\{g\mid\ell_1(g)\le R\}$) is finite and thus any bound of $\ell_1$ (respectively, $\ell_2$) on this set do the job.
\end{proof}

Let's introduce some terminologies and results about sequences of metric spaces.\\
\begin{Def}\label{def:def8}
Let $\XX=((X_n,d_{X_n}))_n$ and $\YY=((Y_n,d_{Y_n}))_n$ be two metric space sequences. 
We say $\XX$ coarsely embeds in $\YY$ if there exists a sequence of maps $\FF=(F_n:X_n\rightarrow Y_n)_n$ and two proper functions $\rho_\pm:\mathbb{R}_+\rightarrow\mathbb{R}_+$ such that
\[\rho_-(d_{X_n}(x,y))\le d_{Y_n}(F_n(x),F_n(x'))\le\rho_+(d_{X_n}(x,x'))\]
for any $n$ and $x, x'\in X_n$.\\
If moreover $(F_n(X_n))_n$ is uniform coarsely dense in $\YY$, i.e if there exists $R>0$ such that $Y_n$ is contain in the $R$-neighborhood  of $F_n(X_n)$ for any $n$, we say that $\FF$ is a coarse equivalence or $\XX$ and $\YY$ are coarsely equivalent.
\end{Def}
In this paper we consider principally embeddings of sequences of metric spaces into Hilbert space, when there no confusion we just say that the sequence \textit{coarsely embeds}.

\begin{Lem}\label{lem:hilb}
Let $\HH_n$ be a sequence of Hilbert spaces. 
The space $\sqcup_n\HH_n$ coarsely embeds into Hilbert space.
\end{Lem}
\begin{proof}
Let's denote $X=\sqcup_n\HH_n$ the coarse disjoint union of $(\HH_n)_n$ and $\rho, p$ are as in Definition \ref{def:coarse}.\\
We consider $\HH=\bigoplus_n\HH_n\oplus\C$ and define the map $F:X\to\HH$ given by the formula: $F(v_n,n)=(v_n)\oplus2^n$.\\
We have that:
\begin{equation}\label{eq:eqlem1}
\|F(v_n)-F(v_m')\|=\left\{\begin{array}{ll}
\|v_n-v_m'\|&\text{if $n=m$}\\
\sqrt{|2^n-2^m|^2+\|v_n\|^2+\|v_m'\|^2}&\text{otherwise}
\end{array}\right.
\end{equation}

Now, let's consider
$$\rho_-(t)=\inf_{\{(v_n,v_m')\in X^{\times 2}\mid t\le d_X(v_n,v_m')\}}\|F(v_n)-F(v_m')\|$$
and 
$$\rho_+(t)=\sup_{\{(v_n,v_m')\in X^{\times 2}\mid d_X(v_n,v_m')\le t\}}\|F(v_n)-F(v_m')\|$$
, respectively, the compression and dilatation functions.\\
By considering separately the case $n=m$ and $n\neq m$, the formula (\ref{eq:eqlem1}) implies that, for all $R\ge 0$, there exists $R'\ge 0$ such that $\|F(v_n)-F(v_m')\|\le R$ implies that $d_X(v_n,v_m')\le R'$.
In other words $\rho_-$ is proper.\\
In the other hand, since $d_X(v_n,v_m')\ge\rho(d_\text{f.a}(p(v_n),p(v_m')))$ for all $v_n,v_m'\in X$, there exists $N_t$ such that $d_X(\HH_n,\HH_m)\ge t+1$ for all $n,m\ge N_t$ with $m\neq n$.\\
It follows that:
$$\rho_+(t)=\sup_{A_t\cup B_t}\|F(v_n)-F(v_m')\|$$
with $A_t=\{(v_n,v_m')\in (\cup_{n\le N_t}\HH_n)^{\times 2}\mid d_X(v_n,v_m')\le t\}$ and $B_t=\{(v_n,v_n')\in (\cup_n\HH_n)^{\times 2}\mid d_X(v_n,v_n')\le t\}$.
Using part $(2)$ of the Definition \ref{def:coarse}, there exists $R_t\ge 0$ such that $d_X(v_n,v_m')\le t$ implies that $v_n\in B_n(0,R_t)\subset \HH_n$ and $v_m\in B_m(0,R_t)\subset \HH_m$ for all $n\neq m$ with $n,m\le N_t$.\\
As a consequence $\rho_+(t)\le\max{(t, \sqrt{2^{2N_t+2}+2R_t^2})}$, in particular $\rho_+$ is well defined.
The following inequalities imply the properness of $\rho_+$:
\begin{align*}
\sup_{B_t}\|F(v_n)-F(v_n')\|&=\sup_{B_t}\|v_n-v_n'\|\\
=t\le\rho_+(t).
\end{align*}
\end{proof}

The following lemma allows us to pass to a language to another concerning embeddings of sequences of metric spaces.
\begin{Lem}\label{lem:lem1}
Let $\XX=(X_n)_n$ be a sequence of metric spaces.
The following are equivalent:
\begin{description}
\item{$(1)$} $\sqcup_n X_n$ coarsely embeds into Hilbert space.
\item{$(2)$} There exist a sequence of maps $\FF=(F_n: X_n\rightarrow \HH_n)_n$ where $\HH_n$ denote a Hilbert space and two proper functions $\rho_\pm:\mathbb{R}_+\rightarrow\mathbb{R}_+$, such that:
$$\rho_-(d_{X_n}(x,y))\le \|F_n(x)-F_n(x')\|\le\rho_+(d_{X_n}(x,x'))$$
for all $n$ and $x, x'\in X_n$.
In other words, $\XX$ coarsely embeds into Hilbert spaces.
\item{$(3)$} For all $R>0$ and $\e>0$, there exists a sequence $(\varphi_n)_n$ of maps $\varphi_n:X_n\rightarrow \HH_n$ with $\HH_n$ a Hilbert space, such that  $\|\varphi_n(x)\|=1$ and:
\begin{description}
\item{$(i)$} for all $n\ge1$ and $ x,y\in X_n$ with $d_{X_n}(x,y)\le R$, $\|\varphi_n(x)-\varphi_n(y)\|\le\e$.
\item{$(ii)$} for all $\e'>0$, there exists $M>0$, $|\langle \varphi_n(x)|\varphi_n(y)\rangle|\le\e'$ whenever $d_{X_n}(x,y)\ge M$.
\end{description}
\end{description}
\end{Lem}
\begin{proof}
$(1)$ implies $(2)$: This is because $X_n$ embeds isometrically into $\sqcup_n X_n$.

$(2)$ implies $(1)$: Let $\FF$, $\rho_\pm$ be as in $(2)$ and $d_X$ the metric on $X=\sqcup_n X_n$.\\
Let $F: (X,d_X,\rho_X,p_X) \to (\HH=\sqcup_n\HH_n,d_\HH,\rho_\HH,p_\HH)$ be the map induced by $\FF$ given by the formula $F(x_n)=F_n(x_n)$.\\
Let's fix $x_{0,n}\in X_n$ for all $n$.
We may assume that $\rho_X=\rho_\HH=\rho$ (otherwise we take the infimum).\\
We have the following inequalities hold:
\begin{equation}\label{eq:lem2}
\begin{aligned}
\rho(d_\text{f.a}(p_X(x_n),p_X(x_m)))&=\rho(d_\text{f.a}(p_\HH\circ F_n(x_n),p_\HH\circ F_m(x_m)))\\
\le d_\HH(F_n(x_n), F_m(x_m))
\end{aligned}
\end{equation}
and
\begin{equation}\label{eq2:lem2}
\rho_-(d_X(x_n,x_n'))\le d_\HH(F_n(x_n),F_n(x_n'))\le\rho_+(d_X(x_n,x_n'))
\end{equation}
for all $x_n,x_n'\in\HH_n$ and $x_m\in\HH_m$.\\
Let's prove the properness of the dilatation function (the proof for the compression follows the same scheme).\\
If $d_X(x_n,x_m)\le R$, in the case $n=m$ by inequalities (\ref{eq2:lem2}) we have $d_\HH(F_n(x_n),F_m(x_m))\le R'=\rho(R)$.
Otherwise $n\neq m$ so by inequalities (\ref{eq:lem2}) $n,m\le N$ where $N$ depends on $R$ and the part $(2)$ of Definition \ref{def:coarse} ensures that $d(x_{0,n},x_n),d(x_{0,m},x_m)\le M$.
Using inequalities (\ref{eq2:lem2}) once again we have $d(F_n(x_{0,n}),F_n(x_n)),d(F_m(x_{0,m}),F_m(x_m))\le M'$.
In other words the dilatation function is proper.
We use the Lemma \ref{lem:hilb} to conclude.\\
$(2)$ implies $(3)$: Take a vector $\xi$ in a Hilbert space $\HH$, then $\exp(\xi)=\sum_{k\ge0}\frac{\xi^{\otimes k}}{n!}\in \bigoplus_{k\ge0}\HH^{\otimes k}$ has a sense and $\langle \exp(\xi)|\exp(\xi')\rangle=e^{\langle\xi|\xi'\rangle}$ for any $\xi,\xi'\in\HH$.\\
Now let's consider $\FF$ and $\rho_\pm$ as in $(2)$ and define $\varphi_n^t(x)=e^{-t\|F_n(x)\|^2}\exp(\sqrt{2t}.F_n(x))$ with $t>0$ which verifies: $\langle\varphi_n^t(x)|\varphi_n^t(y)\rangle=e^{-t\|F_n(x)-F_n(y)\|^2}$. 
This relation implies, for any $t>0$, $\|\varphi_n^t(x)\|=1$ and :
$$e^{-t\rho_+(d_{X_n}(x,y))}\le|\langle\varphi_n^t(x)|\varphi_n^t(y)\rangle|\le e^{-t\rho_-(d_{X_n}(x,y))}$$
for any $x,y\in X_n$.
Since $\rho_+$ is proper, given $\e>0$ we can find $t_0>0$ such that $e^{-t_0\rho_+(r)}\ge1-\frac{\e}{2}$, for any $r\le R$. For such choice of parameter $t$, part $(i)$ of $(3)$ holds with $\varphi_n^{t_0}$. Furthermore because of the properness of $\rho_-$, condition $(ii)$ of $(3)$ is verified for any parameter $t$ fixed. So the implication is proven.

$(3)$ implies $(2)$: Let's take $\e_l=1/l$ and $R_l=\sqrt{l}$. 
Let $(\varphi_n^l)_n$ with $\varphi_n^l:X_n\rightarrow \HH_n^l$ as in $(3)$ for the constants $\e_l$ and $R_l$. 
Let $(M_l)_l$ such that $\|\varphi_n^l(x)-\varphi_n^l(x')\|\ge 1$ whenever $d_{X_n}(x,x')\ge M_l$ that we can assume to be an increasing sequence which goes to infinity.
We fix a point $x_n^0\in X_n$ for any $n$ and define $F_n:X_n\rightarrow \bigoplus_l \HH_n^l$ by $F_n(x_n)=\frac{1}{2}((\varphi_n^1(x_n)-\varphi_n^1(x_n^0))\oplus(\varphi_n^2(x_n)-\varphi_n^2(x_n^0))\oplus\dots)$, then:
$$\rho_-(d_{X_n}(x,x'))\le\|F_n(x)-F_n(x')\|\le d_{X_n}(x,x')+1$$
for any $n$ and $x,x'\in X_n$ with $\rho_-=\frac{1}{2}\sum_{k\ge1}\sqrt{k-1}\chi_{[M_{k-1},M_k)}$ where $\chi_{[M_{k-1},M_k)}$ is the characteristic function of the interval $[M_{k-1},M_k)$.\\
Indeed let $x,x'\in X_n$ and $l$ such that $\sqrt{l-1}\le d_{X_n}(x,x')\le \sqrt{l}$. We have:
\begin{gather*}
\|F_n(x)-F_n(x')\|^2=\frac{1}{4}\sum_{k\le l-1}\|\varphi_n^k(x)-\varphi_n^k(x')\|^2+\frac{1}{4}\sum_{k\ge l}\|\varphi_n^k(x)-\varphi_n^k(x')\|^2\\
\le (l-1)+\frac{1}{4}\sum_{k\ge l}\frac{1}{k^2}\le d(x,x')^2+1
\end{gather*}
Similarly if we take $l$ such that $M_{l-1}\le d(x,x')\le M_l$, then 
$$\|F_n(x)-F_n(x')\|^2\ge\frac{1}{4}\sum_{k\le l-1}\|\varphi_n^k(x)-\varphi_n^k(x')\|^2\ge \frac{l-1}{4}=\rho_-(d_{X_n}(x,x'))^2.$$
\end{proof}

The next proposition can be view as coarse analogue of \textit{the fundamental theorem of geometric group theory} \cite{Fa} for box spaces. 
It finds, in others, applications in the proof of Theorem \ref{thm:thm6} (cf Section \ref{sec:sec5}) which is the cornerstone of Proposition \ref{prop:prop1} and \ref{prop:prop2} but also in Corollary \ref{coro5} which is an important part of Theorem \ref{thm:thm1} proof.

\begin{Prop}\label{prop:prop5} 
Let $G$ be a discrete countable group with a proper left-invariant distance $d$. 
Let $X$ be a proper metric on which $G$ acts freely and geometrically (i.e. properly discontinuously by isometries) and $f:G\rightarrow X$ a $G$-equivariant coarse embedding.\\
If $\NN=(N_n)_n$ is a sequence of normal subgroups of $G$ with trivial intersection, then the sequence of metric groups $(G/{N_n})_n$ coarsely embeds into the sequence $\XX=(X/{N_n})_n$.\\
If moreover $f$ is a coarse equivalence, then $(G/{N_n})_n$ and $\XX$ are coarsely equivalent.
\end{Prop}
\begin{proof}
Let $G$, $X$, $d$, $f$ and $\NN$ be as in the statement. Let $\rho_\pm$ be the two proper functions, that we may assume increasing, such that:
$$\rho_-(d(g,h))\le d_X(f(g),f(h))\le \rho_+(d(g,h))$$
for any $g,h\in G$.\\
Since $f$ is $G$-equivariant, for any $g\in G$, $f(g)=g.x_0$ with $x_0=f(e)$. 
Moreover $f$ induces a sequence of maps $f_n: G/{N_n}\rightarrow X_n$ (since $N_n\unlhd G$ is normal, $gN_n=N_ng$ for any $g\in N_n$) which is nothing but the well defined map $f_n(\overline{x})=\overline{f(x)}$.\\
The distance $d_{X,n}$ on $X/{N_n}$ is given by the formula $d_{X,n}(\overline{x},\overline{y})=\inf_{h,h'\in N_n}d_X(h.x,h'.y)$.
Let's define $\ell_{x_0,n}: G/{N_n}\rightarrow\R_+$ by the formula $\ell_{x_0,n}(\overline{g})=\inf_{h\in N_n}d_X(gh.x_0,x_0)=d_{X,n}(\overline{g.x_0},\overline{x_0})$ and denote $\ell_n$ the quotient length induced by $d$ on $G/{N_n}$.
Since $G$ acts isometrically it is enough for our purpose to prove the existence of two proper functions $\rho'_\pm$ such that:
\begin{equation}\label{eq:eq1}
\rho'_-(\ell_n(\overline{g}))\le \ell_{x_0,n}(\overline{g})\le \rho'_+(\ell_n(\overline{g})).
\end{equation}
for any $n$ and $\overline{g}\in G/{N_n}$.

Let's prove the left side of the inequality (\ref{eq:eq1}). 
Because $\rho_-$ is proper and does not depend on $n$, for any $R>0$, $g\in G$ with $\ell_{x_0,n}(\overline{g})\le R$ and $\e>0$, there exists $R'$ independent of $n$ such that: any $h\in N_n$ with 
$$\rho_-(\ell(gh))\le d_X(gh.x_0,x_0)\le \ell_{x_0,n}(\overline{g})+\e \le R+\e$$
is contained in a ball with center $g^{-1}$ and radius $R'$ which is a finite set since $\ell$ is proper. 
In particular for such $h\in N_n$, $\ell_n(\overline{g})\le\ell(gh)\le R'$.
In other words, for any $R>0$, there exists $R'>0$ such that for any $n$, $\ell_{x_0,n}(\overline{g})\le R$ implies $\ell_n(\overline{g})\le R'$.

For the right side of inequality (\ref{eq:eq1}) we do something similar. Since $\ell$ is proper, there exists $h_0\in N_n$ such that, $\ell_n(\overline{g})=\ell(gh_0)$ and so
$$\ell_{x_0,n}(\overline{g})\le d_X(gh_0.x_0,x_0)\le\rho_+(\ell(gh_0))=\rho_+(\ell_n(\overline{g}))$$
The second part of the proposition is because the R-density of $f(G)\subset X$ is preserved by taking quotient.
\end{proof}

Here are some consequences of Proposition \ref{prop:prop5}.
\begin{Cor}\label{cor:cor1}
Let $G$ be a discrete residually finite group and $\ell_i$, $i=1,2$, be two proper lengths on $G$.\\
Let $\NN=(N_n)_n$ be a sequence of finite index normal subgroups of $G$ with trivial intersection, then $\Box_{\NN,1} G$ and $\Box_{\NN,2} G$ are coarsely equivalent with dilation and compression functions which depend only on $\ell_1$ and $\ell_2$.
By $\Box_{\NN,i} G$ we mean the box space obtained by taking the quotient metric associated to $\ell_i$, $i=1,2$, on each $G/{N_n}$.
\end{Cor}
\begin{proof}
We apply proposition \ref{prop:prop5} to the identity map $\text{Id}:(G,\ell_1)\rightarrow (G,\ell_2)$ which is a $G$-equivariant coarse equivalence by Proposition \ref{prop:prop4}.\\
\end{proof}

\begin{Cor}\label{cor:cor2}
Let $G$ be a discrete residually finite group and $H\le G$ a finite index subgroup of $G$.
Suppose there exists $\NN'=(N'_n)_n$ a sequence of finite index normal subgroups of $H$ with trivial intersection such that $\Box_{\NN'} H$ embeds into a Hilbert space (or any other metric space). Then $G$ also admits such a sequence.
\end{Cor}
\begin{proof}
Let's fix a proper length $\ell$ on $G$. The inclusion map $\text{Id}|_H: (H,\ell|_H)\rightarrow (G,\ell)$ is a H-equivariant isometry so Proposition \ref{prop:prop5} can apply (remember that, $\ell|_H$ is equivalent to any other proper length on $H$ by Proposition \ref{prop:prop4}).
\end{proof}

As in the classical case, we have heredity on the uniform embedding of box spaces.
\begin{Prop}\label{prop:prop6}
Let $G$ be a discrete residually finite group and $H\le G$ a subgroup of $G$.
Suppose there exists $\NN=(N_n)_n$ a sequence of finite index normal subgroups of $G$ with trivial intersection such that $\Box_\NN G$ embeds into a Hilbert space (or any other metric space), then $H$ also admits such a sequence.
\end{Prop}
\begin{proof}
$\NN\cap H=(H\cap N_n)_n$ defines a sequence of finite index normal subgroups of $H$ with trivial intersection. If we endow $H$ with the restriction of a given proper length $\ell$ on $G$ then $\Box_{\NN\cap H}H$ embeds isometrically into $\Box_\NN G$.
\end{proof}
\section{Finite quotients of a free product}\label{sec:sec4}
This section, devoted to the proof of Theorem \ref{thm:thm1}, is organized as follows:\\
We start by a succinct introduction to Bass-Serre theory \cite{Ser} that we need in the rest and which allows us to formulate and prove the Proposition \ref{prop:prop7bis}.\\
After that we present a process to obtain faithful sequences of finite quotients of a given free product of residually finite groups with an interesting decomposition property.
As mentioned in the outline, this construction and the sequences that it produces are in the heart of the proof of Theorem \ref{thm:thm1}.\\
In Subsection \ref{sub:sub41} we introduce gluing techniques and give a reduction of Theorem \ref{thm:thm1} to the uniform embedding of a certain sequence of large girth subgroups endowed with particular metrics. This is Theorem \ref{thm:thm3}.\\
We finish Theorem \ref{thm:thm1} proof in Subsection \ref{sub:sub42} which is dedicated to the study of this large girth sequence.
In this subsection we introduce direct limit formalism for sequence of groups and approximate our sequence (cf. Proposition \ref{prop:prop8}) with another natural metric structure on which we have a better grasp.\\

Let $A$, $B$ be two discrete countable groups and $G=A\star B$ their free product. 
We can associate to $G$ a tree, called \textbf{Bass-Serre tree} \cite{Ser}, $T$ given by:
\begin{gather*}
V(T)=G/A\cup G/B\\
E(T)=G
\end{gather*}
with the endpoint maps, the projections $G\rightarrow G/A$ and $G\rightarrow G/B$.

The action of $G$ on its cosets is combinatorial with respect to $T$ (i.e it respects the adjacency) and defines an isometric action on $T$ which has two orbits and stabilizers equal to a conjugate of $A$ or $B$ depending on the family of cosets we consider.\\
We can deduce from this a simple criterion to characterize normal subgroups of $G$ which act freely on the Bass-Serre tree $T$: a normal subgroup $H\unlhd G$ acts freely on $T$ if and only if $H$ has trivial intersection with $A$ and $B$.\\

The following lemma is a \textit{normal form} type result on quotients of a free product:
\begin{Lem}\label{lem:d}
Let $A$ and $B$ be two discrete countable groups and $G=A\star B$ their free product.\\
Given a quotient $Q$ of $G$, $A$ (respectively, $B$) factorizes through a quotient, denoted $\overline{A}$ (respectively, $\overline{B}$), such that $\overline{A}$ and $\overline{B}$ are subgroups of $Q$ and :
$$Q\simeq\overline{A}\star\overline{B}/F$$
where $F\unlhd\overline{A}\star\overline{B}$ is a free normal subgroup.\\
Furthermore if $Q$ is finite then $\overline{A}$, $\overline{B}$ and $\text{rk}(F)$ are finite. 
\end{Lem}
\begin{proof}
Let $\pi: G\rightarrow Q$ be the canonical projection. 
If we choose $\overline{A}=A/{\ker(\pi|_B)}$ and $\overline{B}=B/{\ker(\pi|_B)}$, $\pi$ factor through the following diagram:
$$\xymatrix{
G\ar[r]^\pi\ar@{->>}[d]&
Q\\
\overline{A}\star\overline{B}\ar[ru]^{\overline{\pi}}
}$$
Moreover, since $\overline{\pi}|_{\overline{A}}$ (respectively, $\overline{\pi}|_{\overline{B}}$) has trivial kernel by construction, $\overline{A}$ and $\overline{B}$ are subgroups of $Q$. 
In particular, if $Q$ is finite then $\overline{A}$ and $\overline{B}$ are finite.\\
The normal subgroup $F=\ker(\overline{\pi})$ of $\overline{G}=\overline{A}\star\overline{B}$ does not intersect $\overline{A}$ and $\overline{B}$.
So $F$ acts freely on the Bass-Serre tree $\overline{T}$ of $\overline{G}$ and this implies that $F$ is a free group \cite{Ser}. 
Moreover if $Q$ is finite the action of $F$ on $\overline{T}$ is cocompact. Since $\overline{T}$ has bounded geometry (the geometry of $\overline{T}$ is controlled by the cardinality of $\overline{A}$ and $\overline{B}$ which are finite), these two facts  imply that $F$ has finite rank \cite{Ser}.
\end{proof}

Let $Y$ be a graph. 
It is well known that any maximal tree subgraph of $Y$ contains every vertices of $Y$.\\
Given a group $G$ which acts simplicialy on $X$, a \textbf{tree representation} of $X/G$ is a connected subgraph $X^*$ of $X$ which is mapped isomorphically by the projection to a maximal tree subgraph of $X/G$.
Note that such a representation always exists.

\begin{Prop}\label{prop:prop7bis}
Let $A$ and $B$ be two discrete countable groups and $G=A\star B$ their free product. Let $D$ be the kernel of the morphism $\pi: G\rightarrow A\times B$.
Then $D$ has basis $\BB=\{[a,b]\mid a\in A,b\in B; a,b\neq1\}$ and there exist constants $C,\lambda>0$ independent of $A$ and $B$ such that $(D,\BB)$ is $(C,\lambda)$-quasi-isometric to the Bass-Serre tree associated to $G$.
\end{Prop}
\begin{proof}
An orientation on the Bass-Serre tree $T$ of $G$ which is preserved by the action corresponds to a choice of an orientation of the quotient $G\backslash T$. 
In our case there is only two choices and we assume here that oriented edges on $T$ goes from $B$-cosets to $A$-cosets.\\
The subgroup $D$ is normal and does not intersect $A$ and $B$, it follows that $D$ acts freely on the tree $T$ and so is a free group \cite{Ser}.\\
Let $T'=D\backslash T$ be the quotient $T$ by the action of $D$, then $T'$ is a complete bipartite graph given by:
$$V(T')={(A\times B)}/A\sqcup {(A\times B)}/B$$
$$E(T')=A\times B$$
with endpoint maps given by the quotient maps on, respectively, $A$ and $B$.\\
In particular a tree representation of $T'$ is given by the subtree $T_r$ of $T$ with:
$$V(T_r)=\{aB; a\in A\}\sqcup\{bA; b\in B\}\subset V(T)$$
$$E(T_r)=A\cup B\subset G=E(T)$$
Let $\partial_V T_r^+$ be the set of vertices of $T\setminus T_r$ which are target of a edge with source a vertices of $T_r$. Then the set $\BB=\{g'\neq 1\mid g'\in D, g'.T_r\cap\partial_V T_r^+\neq\emptyset \}$ which is, by direct computation, equal to $\{[a,b]; a\in A,b\in B\}$ forms a basis of the free group $D$ (this geometric criterion can be found Section 3 \cite{Ser}). 
It appears that the rank of $D$ is finite when $A$ and $B$ are finite.\\

Since $G$ acts isometrically, to prove the last part of the proposition it is enough to prove the following inequalities:
\begin{equation}\label{eq:eq3}
\frac{1}{4}.d_T(g.A,A)\le \ell(g)\le d_T(g.A,A)+1
\end{equation}
with $g\in D$. 
These will imply that $f:D\rightarrow T$ with $f(g)=g.A$ is a $(4,1)$-quasi-embedding but $T'$ has diameter equal to 2, in other words, $f(G)$ is $2$-dense and by consequence defines a $(4,1)$-quasi-isometry.

Since $d_T(s.A,A)=4$ for any $s\in \SSS=\BB\cup\BB^{-1}$, given $g=s_0\dots s_n\in D$ a reduced word with $s_i\in\SSS$ and $n\ge1$, we have :
$$d(g.A,A)\le \sum_{i=0}^{n-1} d(s_0\dots s_i.A, s_0\dots s_{i+1}.A)=\sum_{i=0}^{n-1} d(A, s_{i+1}.A)=4.\ell(g)$$
This prove the left side of the inequalities (\ref{eq:eq3}).\\
For the right side, let $g\in D$ and $x_0=A,\dots,x_n=g.A$ a sequence of vertices which describes the geodesic path in $T$ between $A$ and $gA$.
Since $D.T_r=T$, we can find $g_0=e, g_1,\dots, g_n=g$ such that $x_i\in g_i.T_r$ for any $i=0,\dots,n$. 
If we denote $s_0=g_1$ and $s_i=g_i^{-1}g_{i+1}$ for $1\le i\le n-1$, then $g=s_0.\dots.s_{n-1}$ and:
$$d_T(s_iT_r,T_r)=d_T(g_{i+1}.T_r,g_i.T_r)\le d(x_{i+1},x_i)=1$$
so $s_i\in\SSS$ for any $i$.\\
Indeed, if we denote $\partial_V T_r^-$ the analogue of $\partial_V T_r^+$ for the opposite orientation on $T$.
An element $g'\in D$ belongs to $\SSS$ if and only if $g'.T_r\cap\partial_V T_r^+\neq\emptyset$ or $g'.T_r\cap\partial_V T_r^-\neq\emptyset$.
A consequence is that $d(g'T_r,T_r)\ge2$ whenever $g'\in D\setminus\SSS$. 
It follows that $\ell(g)\le n$.
\end{proof}

Let us consider $A$, $B$ two discrete countable groups as above and $\overline{A}$, $\overline{B}$ two quotients of respectively $A$ and $B$. 
Let $G=A\star B$ (respectively, $\overline{G}=\overline{A}\star\overline{B}$) their free product. 
We have the following commutative diagram:
$$\xymatrix{
1\ar[r]&
D\ar[r]\ar[d]&
G\ar[r]\ar[d]&
A\times B\ar[r]\ar[d]&
1\\
1\ar[r]&
D'\ar[r]\ar[d]&
\overline{G}\ar[r]\ar[d]&
\overline{A}\times \overline{B}\ar[r]\ar[d]&
1\\
&1&1&1&\\
}$$
Note that, the basis $\BB$ of $D$ described in Proposition \ref{prop:prop7bis} is sent to the basis $\BB'$ of $D'$ defined analogously by the quotient map.\\
In the rest of the text we will talk about \textbf{commutator subgroups} to designate kernel's of type $D^{(\cdot)}$.\\

Let $\Gamma$ be a residually finite group.
Recall that a sequence of finite quotients of $\Gamma$ is called \textbf{faithful} if any non-trivial element $g\in \Gamma$ acts non-trivially by left translation on all but a finite number of these quotients.\\
Now, if we assume $A$ and $B$ residually finite, it is well know that their free product $G=A\star B$ is residually finite \cite{BT} (the construction described in that paper is recall and used in the next section).\\
Here we propose another construction.
Let $(\overline{A}_n)_n$ and $(\overline{B}_n)_n$ be two faithful sequence of finite quotients of respectively $A$ and $B$. 
We denote $G_n=\overline{A}_n\star\overline{B}_n$ and $D_n$ the commutator subgroup of $G_n$ which has finite rank since $\overline{A}_n$ and $\overline{B}_n$ are finite as explained above.
Let $(M_{(n,k)})_k$ be a sequence of finite index characteristic subgroups of $D_n$, i.e $M_{(n,k)}$ is normalized by each automorphism of $D_n$.
Moreover we assume that the subgroups $(M_{(n,k)})_k$ have trivial intersection and that $M_{(n,k+1)}$ is a proper characteristic subgroup of $M_{(n,k)}$.\\
For example we can take $M_{(n,k+1)}=\langle \{w^s\mid w\in M_{(n,k)}\}\cup\{[w,w']\mid w,w'\in M_{(n,k)}\}\rangle$ with $s\ge 2$ fixed integer. 
Since $D_n$ has finite rank it is not hard to see that $M_{(n,k)}$ has finite index in $D_n$ and Proposition $3.3$ of \cite{LS} ensures that the intersection is trivial.
\begin{Prop}\label{prop:prop8bis}
The sequence of finite quotients $(G_n/{M_{(n,k)}})_{n,k\ge n}$ of $G$ is faithful.
\end{Prop}
\begin{proof}
Let $E$ be a finite symmetric set of $G$ with $e\in E$.
Let $E'=E^{(2)}=\{x.y\mid x,y\in E\}\subset G$ and $E_A\subset A$ (respectively, $E_B\subset B$) be the set of letters in $A$ (respectively, $B$) which appears in the normal form of the elements of $E'$.
Since $E$ is supposed to be finite, $E'$, $E_A$ and $E_B$ are finite.
In the other hand $A$ and $B$ are supposed to be residually finite so there exists $N_0$ such that $E_A$ (respectively, $E_B$) is sent bijectively into $\overline{A}_n$ (respectively, $\overline{B}_n$) for any $n\ge N_0$, in particular $E'$ is sent bijectively into $G_n$ for $n\ge N_0$.\\
If we denote $D$ the commutator subgroup of $G$ then $D\cap E'\neq 1$ implies $D_n\cap \overline{E'}\neq 1$ (where $\overline{E'}=\pi_n(E)$ and $\pi_n:G\rightarrow G_n$ is the quotient map) for any $n\ge N_0$.
Moreover, $d_\BB(e,D\cap E')\ge d_{\BB_n}(e,D_n\cap\overline{E'})$ where $\BB$ (respectively, $\BB_n$) denote the basis of Proposition \ref{prop:prop7bis} and $d_{(*)}$ the word distance associated.
Since the sequence $(M_{(n,k)})_k$ is such that $M_{(n,k+1)}$ is proper and characteristic inside of $M_{(n,k)}$ Proposition $3.2$ of \cite{LS} ensures that $d_{\BB_n}(M_{(n,k)},e)\ge k$ in $D_n$.
So we have that for any $k\ge n\ge\max\{N_0, d_\BB(e,D\cap E')+1\}$ each elements of $E'$ are sent non-trivially into the finite quotients $G_n/{M_{(n,k)}}$.\\
Assume there exist $x,y\in E$, $x\neq y$ with same image into $G_n/{M_{(n,k)}}$, then $x.y^{-1}\in E'$ with $x.y^{-1}\neq e$ is sent trivially into $G_n/{M_{(n,k)}}$ which is a contradiction. So $E$ is sent bijectively into $G_n/{M_{(n,k)}}$ for any $k\ge n\ge\max\{N_0,d_\BB(e,D\cap E')+1\}$.
\end{proof}

In order to prove Theorem \ref{thm:thm1} with a sequence of the above type it seems that we must find suitable $(\overline{A}_n)_n$, $(\overline{B}_n)_n$ and $(M_{(n,k)})_k$'s.
As it is explained in the next subsection, if we assume $A$ and $B$ amenable the choice of $(\overline{A}_n)_n$ and $(\overline{B}_n)_n$ is not matter.

\subsection{Geometry of metric spaces and gluing techniques}\label{sub:sub41}
In this subsection we develop gluing techniques that allow us, under the assumption that $A$ and $B$ are amenable, to reduce the uniform embedding of the sequence $(G_n/{M_{(n,k)}})_{n,k\ge n}$ introduced above to that of the large girth sequence of subgroups $(D_n/{M_{(n,k)}})_{n,k\ge n}$ endowed with the restricted metric structure.\\

Let us first introduce a variant of property (A) that fits better with our purpose:
\begin{Def}\label{def:def10}
A sequence $\XX=(X_n)_n$ of metric spaces is equi-exact or uniformly exact if for any $R>0$, $\varepsilon>0$ and $n\ge0$, there exists a partition of unity $(\varphi_{(n,i)})_{i\in I_n}$ of $X_n$ subordinated to a cover $\UU_n=(U_{(n,i)})_{i\in I_n}$, i.e. $\{\varphi_{(n,i)}\neq0\}\subset U_{(n,i)}$, such that:
\begin{description}
\item{$(1)$} for any $n\ge1$ and $x,y\in X_n$, $d_{X_n}(x,y)\le R$ implies $\sum_{i\in I_n}|\varphi_{(n,i)}(x)-\varphi_{(n,i)}(y)|\le\varepsilon$.
\item{$(2)$} the cover $\UU=\{\UU_n\}_n$ of $\XX$ is uniformly bounded, i.e there exists $C>0$ such that $diam(U_{(n,i)})\le C$ for any $n\ge 0$ and $i\in I_n$.
\end{description}
\end{Def}

\begin{Lem}\label{prop:prop11}
Let $G$ be a discrete residually finite group, $\ell$ a proper length on $G$ and $\NN=(N_n)_n$ a sequence of finite index normal subgroups with trivial intersection.
If $G$ is amenable, then the sequence of metric spaces $\XX=(X_n)_n$, where $X_n=G/{N_n}$ is endowed with the quotient metric induced by $\ell$, forms a equi-exact sequence.
\end{Lem}
\begin{proof}
Since $G$ is amenable by proposition 11.39 \cite{Roe} $\Box_\NN G$ has property (A). 
Let $R>0$, $\e>0$ be given. 
Let $\xi=(\xi_i)_{i\in I}: \Box_\NN G\rightarrow\ell^2(I)$ with $M>0$ as in Definition \ref{def:def2}.\\
Define $\varphi_{(n,i)}: X_n\rightarrow \R_+$ by $\varphi_{(n,i)}=(\xi_i|_{X_n})^2$ and $U_{(n,i)}=\{\varphi_{(n,i)}\neq0\}$. 
Since $\text{supp($\xi(x)$)}\subset B(x,M)$, $U_{(n,i)}\subset B(x,M)$ and then $\{\{U_{(n,i)}\}_{i\in I_n}\}_n$ forms an uniformly bounded cover of $\XX$.
Moreover $\sum_{i\in I_n}\varphi_{(n,i)}(x)=\|\xi(x)\|^2=1$ and for any $x,y\in X_n$ such that $d_{X_n}(x,y)=d_{\Box_\NN G}(x,y)\le R$:
$$\sum_{i\in I_n}|\varphi_{(n,i)}(x)-\varphi_{(n,i)}(y)|\le\|\xi(x)-\xi(y)\|\le \e.$$
\end{proof}

\begin{Prop}\label{prop:prop13}
Let $\XX=(X_n)_n$ be a sequence of metric spaces. Assume that for any $R>0$, $\e>0$ and $n\ge 1$, there exists a partition of unity $(\varphi_{(n,i)})_{i\in I_n}$ of $X_n$ such that:
\begin{description}
\item{$(1)$} for any $x,y\in X_n$, $d_{X_n}(x,y)\le R$ implies $\sum_{i\in I_n}|\varphi_{(n,i)}(x)-\varphi_{(n,i)}(y)|\le\e$.
\item{$(2)$} the cover $\UU=\{U_{(n,i)}\}_{n,i\in I_n}$ with $U_{(n,i)}=\{\varphi_{(n,i)}\neq0\}$ of $\XX$ is uniformly embeddable into a Hilbert space.
\end{description}
Then $\XX$ is uniformly embeddable.
\end{Prop}
\begin{proof}
We prove that $\XX$ verifies the criterion $(3)$ of Lemma \ref{lem:lem1}.\\
Let's fix $R>0$, $\e>0$ and consider $(\varphi_{(n,i)})_{i\in I_n}$ as above.
Since $\UU$ is uniformly embeddable, $\UU(R)=\{U_{(n,i)}(R)\}_{n,i\in I_n}$ is also embeddable (these two families are uniformly quasi-isometric). 
So we can find $\xi_{(n,i)}: U_{(n,i)}(R)\rightarrow \HH_{(n,i)}$ as in Lemma \ref{lem:lem1} characterization $(3)$ for the constants $R>0$ and $\e>0$.\\ 
Let's define $\varphi_n: X_n\rightarrow \bigoplus_{i\in I_n}\HH_{(n,i)}$ by the formula $\varphi_n(x)=(\varphi_{(n,i)}(x)^{1/2}\xi_{(n,i)}(x))_{i\in I_n}$. 
It is immediate from the construction that $\|\varphi_n(x)\|=1$ for any $x\in X_n$.\\
Let $x,y\in X_n$ with $d_{X_n}(x,y)\le R$, there exists $i\in I_n$ such that $x\in U_i^n$ and so $x,y\in U_{(n,i)}(R)$.
In particular $\|\xi_i^n(x)-\xi_i^n(y)\|\le\e$. 
It follows:
\begin{gather*}
\|\varphi_n(x)-\varphi_n(y)\|=[\sum_{i\in I_n}\|\varphi_{(n,i)}(x)^{1/2}(\xi_{(n,i)}(x)-\xi_{(n,i)}(y))+\\
(\varphi_{(n,i)}(x)^{1/2}-\varphi_{(n,i)}(y)^{1/2})\xi_{(n,i)}(y)\|^2]^{1/2}\le\\
 [\sum_{i\in I_n}\varphi_{(n,i)}(x).\|\xi_{(n,i)}(x)-\xi_{(n,i)}(y)\|^2]^{1/2}+[\sum_{i\in I_n}|\varphi_{(n,i)}(x)^{1/2}-\varphi_{(n,i)}(y)^{1/2}|^2\|\xi_{(n,i)}(y)\|^2]^{1/2}\le\\
 \e +[\sum_{i\in I_n}|\varphi_{(n,i)}(x)^{1/2}-\varphi_{(n,i)}(y)^{1/2}|^2]^{1/2}\le\e +[\sum_{i\in I_n}|\varphi_{(n,i)}(x)-\varphi_{(n,i)}(y)|]^{1/2}\le \e+\e^{1/2}.
\end{gather*}
In other words, for any $x,y\in X_n$, $d_{X_n}(x,y)\le R$ implies $\|\varphi_n(x)-\varphi_n(y)\|\le\e'=\e+\e^{1/2}$.
The second part to verify is a simple consequence of the following inequalities:
\begin{gather*}
|\langle \varphi_n(x)|\varphi_n(y)\rangle|\le\sum_{i\in I_n}\varphi_i^n(x)^{1/2}\varphi_i^n(y)^{1/2}|\langle\xi_i^n(x)|\xi_i^n(y)\rangle|\le\\
\sup_{n,i\in I_n}\sup\{|\langle\xi_i^n(x)|\xi_i^n(y)\rangle|\}\mid d(x,y)\ge M, x,y\in X_n\}
\end{gather*}
\end{proof}

\begin{Prop}\label{prop:prop13bis}
Let $\XX=(X_n)_n$ be a sequence of metric spaces.
Suppose for any $R>0$, $\e>0$ and $n\in\N$, there exists a partition of the unity $(\varphi_{(n,i)})_{i\in I_n}$ of $X_n$ such that:
\begin{description}
\item{$(1)$} for any $x,y\in X_n$, $d_{X_n}(x,y)\le R$ implies $\sum_{i\in I_n}|\varphi_{(n,i)}(x)-\varphi_{(n,i)}(y)|\le\e$.
\item{$(2)$} the family $\UU=\{U_{(n,i)}\}_{n,i\in I_n}$ with $U_{(n,i)}=\{\varphi_{(n,i)}\neq0\}$ is equi-exact.
\end{description}
Then $\XX$ is equi-exact.
\end{Prop}
\begin{proof}
Let $R>0$, $\e>0$ be given.
Let $(\varphi_{(n,i)})_{n,i\in I_n}$ be the maps subordinated to $\UU=\{U_{(n,i)}\}_{n,i\in I_n}$ as in the statement.\\
Since $\UU$ is equi-exact so is $\UU(R)=\{U_{(n,i)}(R)\}_{n,i\in I_n}$. In particular we can find $\psi_{(n,i,j)}:X_n\rightarrow \R_+$ subordinated to $V_{(n,i,j)}\subset U_{(n,i)}(R)$ with $j\in J_{(n,i)}$ such that for any $x,y\in U_{(n,i)}(R)$ with $d_{X_n}(x,y)\le R$, $\sum_{j\in J_{(n,i)}}|\psi_{(n,i,j)}(x)-\psi_{(n,i,j)}(y)|\le\e$ and $\VV=\{V_{(n,i,j)}\}_{n,i\in I_n,j\in J_{(n,i)}}$ is uniformly bounded.\\
Let's define the partition of the unity of $X_n$, $\eta_{(n,i,j)}=\varphi_{(n,i)}\psi_{(n,i,j)}$ subordinated to $\VV$. 
Moreover,
\begin{gather*}
\sum_{i\in I_n, j\in J_{(n,i)}}|\eta_{(n,i,j)}(x)-\eta_{(n,i,j)}(y)|\le 
\sum_{i\in I_n}\varphi_{(n,i)}(x)\sum_{j\in J_{(n,i)}}|\psi_{(n,i,j)}(x)-\psi_{(n,i,j)}(y)|\\
+\sum_{i\in I_n}|\varphi_{(n,i)}(x)-\varphi_{(n,i)}(y)|\sum_{j\in J_{(n,i)}}\psi_{(n,i,j)}(x)\\
\le \sum_{i\in I_n}\varphi_{(n,i)}(x)\sum_{j\in J_{(n,i)}}|\psi_{(n,i,j)}(x)-\psi_{(n,i,j)}(y)|
+\sum_{i\in I_n}|\varphi_{(n,i)}(x)-\varphi_{(n,i)}(y)|
\end{gather*}

Since $\varphi_{(n,i)}(x)\neq 0$ implies that $x\in U_{(n,i)}$ and then $y\in U_{(n,i)}(R)$, by definition of $(\psi_{(n,i,j)})_{j\in J_{(n,i)}}$, $\sum_{j\in J_{(n,i)}}|\psi_{(n,i,j)}(x)-\psi_{(n,i,j)}(y)|\le \e$.\\
Moreover for any $x,y\in X_n$ with $d_{X_n}(x,y)\le R$ we know that: $\sum_{i\in I_n}|\varphi_{(n,i)}(x)-\varphi_{(n,i)}(y)|\le\e$. 
Combining these two estimations we have: $\sum_{i\in I_n, j\in J_{(n,i)}}|\eta_{(n,i,j)}(x)-\eta_{(n,i,j)}(y)|\le2\e$ as desired.
\end{proof}

\begin{Cor}\label{cor:cor4}
Let $\XX=(X_n)_n$ and $\YY=(Y_n)_n$ be two sequences of metric spaces and $\FF=(F_n:X_n\rightarrow Y_n)_n$ a sequence of contractive maps.
Assume that $\YY$ is equi-exact and for any uniformly bounded cover $\UU=\{\{U_{(n,i)}\}_{i\in I_n}\}_n$ of $\YY$, i.e $\{U_{(n,i)}\}_{i\in I_n}$ is a partition of $Y_n$ and $\sup_{n, i\in I_n}diam(U_{(n,i)})<+\infty$, the cover $\FF^{-1}(\UU)=\{\{F_n^{-1}(U_{(n,i)})\}_i\}_n$ of $\XX$ is uniformly embeddable (respectively, equi-exact). 
Then the family $\XX$ is uniformly embeddable (respectively, equi-exact).
\end{Cor}
\begin{proof}
Let $R>0$, $\varepsilon>0$ be two positive constants. 
Since $\YY$ is an equi-exact sequence, we can find an uniformly bounded cover $\UU$ of $\YY$ and a subordinated partition of unity $\{(\varphi_{(n,i)})_{i\in I_n}\}_n$ which satisfies Definition \ref{def:def10} with respect to the constants $R$ and $\e$.
Then $\{(\varphi_i^n\circ F_n)_{i\in I_n}\}_n$ is a partition of unity on the sequence $\XX$ subordinated to $\FF^{-1}(\UU)$ which is uniformly embeddable (respectively, equi-exact) by assumption. Proposition \ref{prop:prop13} (respectively, Proposition \ref{prop:prop13bis}) can be applied.
\end{proof}

\begin{Cor}\label{cor:cor5}
Let $\XX=(X_n)_n$ and $\YY=(Y_n)_n$ be two sequences of metric spaces, $G$ a group which acts by isometries on each $X_n$ and $Y_n$ for any $n$ and $\FF=(F_n:X_n\rightarrow Y_n)_n$ a sequence of $G$-equivariant contractions.
Assume that $\YY$ is equi-exact and the actions on $\YY$ are uniformly cobounded, i.e there exists $C>0$ such that $\text{diam$(Y_n/G)$}\le C$ for any $n$.\\
If for any $n$, there exists $y_n\in Y_n$ such that for any $R\ge0$ the sequence of inverse images $(F_n^{-1}(B(y_n,R)))_n\subset \XX$ is uniformly embeddable (respectively, equi-exact) then $\XX$ is uniformly embeddable (respectively, equi-exact).
\end{Cor}
\begin{proof}
Let $R>0$, $\varepsilon>0$ be two positive constants. 
Since $\YY$ is an equi-exact sequence, we can find an uniformly bounded cover $\UU$ of $\YY$ and a subordinated partition of unity $\{(\varphi_{(n,i)})_{i\in I_n}\}_n$ which satisfies Definition \ref{def:def10} with respect to the constants $R$ and $\e$.
Let $y_n\in X_n$ as in the statement.
Since the actions of $G$ are uniformly cobounded, we can find a sequence of subspaces $B_n\subset Y_n$ such that $y_n\in B_n$, every orbit of the $G$-action on $Y_n$ intersects $B_n$ and $(B_n)_n$ is a uniformly bounded sequence of metric spaces.\\
Let $R_0=\sup_{i,n}\text{diam}(U_i^n)+\sup_n \text{diam}(B_n)<+\infty$, we claim that for any $n\in \N$ and ${i\in I_n}$, there exists $g_i^n$ such that $g_i^nU_i^n\subset B(y_n,R_0)$. 
Indeed, let $g_i^n\in G$ be such that $g_i^nU_i^n$ intersects $B_n$ and $z\in s_i^nU_i^n\cap B_n$, then for any $y\in U_i^n$ we have :
\[d(g_i^ny,y_n)\le d(g_i^ny,z)+d(z,y_n)\le  d(y,(g_i^n)^{-1}z)+d(z,y_n)\le R_0.\]
It follows $g_i^nF_n^{-1}(U_i^n)=F_n^{-1}(g_i^nU_i^n)\subset F_n^{-1}(B(y_n, R_0))$ which is uniformly embeddable in $n$ (respectively, equi-exact). 
So Corollary \ref{cor:cor4} can be applied
\end{proof}

Now we can formulate our reduction result discussed at the beginning of this subsection:\\
Let $A$, $B$ be two discrete residually finite groups and $(\overline{A}_n)_n$, $(\overline{B}_n)_n$ two faithful sequence of finite quotients of respectively $A$ and $B$.\\
Let's denote $G=A\star B$, $G_n=\overline{A}_n\star\overline{B}_n$ their free products and $D\unlhd G$, $D_n\unlhd G_n$ the free commutator subgroups associated.\\
Let $(M_{(n,k)})_k$ be sequence of finite index characteristic subgroups of $D_n$ with trivial intersection (we need a bit less assumptions here because we do not require faithfulness of the sequence of quotients).
Assume $A$ and $B$ are amenable and fix a proper length function $\ell$ on $G$.
\begin{Thm}\label{thm:thm3}
If the sequence of metric spaces $(D_n/{M_{(n,k)}})_{n,k\ge n}$, where each $D_n/{M_{(n,k)}}$ is endowed with the restriction of the quotient length on $G_n/{M_{(n,k)}}$ induced by $\ell$, embeds uniformly into a Hilbert space, then $(G_n/{M_{(n,k)}})_{n,k\ge n}$ embeds uniformly.
\end{Thm}
\begin{proof}
Any free product $\Gamma=\Gamma_1\star \Gamma_2$ is a subgroup of $(\Gamma_1\times \Gamma_2)\star\Z/2$.
Indeed, let $u$ be the generator of $\Z/2$, then the homomorphism $f:\Gamma\rightarrow (\Gamma_1\times\Gamma_2 )\star\Z/2$ which sends $g_1\in \Gamma_1\mapsto g_1$ and $g_2\in \Gamma_2\mapsto u.g_2.u$ is a realization.\\
Moreover, if we assume $\Gamma_1$ and $\Gamma_2$ amenable, stability of amenability by direct product ensures that $\Gamma_1\times \Gamma_2$ is amenable \cite{BrOz} and heredity on box space uniform embedding (cf Proposition \ref{prop:prop6}) allows us to assume $B=\Z/2$.

Let $\YY=(Y_n)_n$ with $Y_n=G_n/{D_n}=\overline{A}_n\times\Z/2$. 
We have the following exact sequence:
$$1\rightarrow D_n/{M_{(n,k)}}\rightarrow G_n/{M_{(n,k)}}\rightarrow \overline{A}_n\times\Z/2\rightarrow 1.$$
$G$ acts on both $X_n$ and $Y_n$ transitively which ensures the action to be uniformly cobounded and the projections $\pi_n: X_n\rightarrow Y_n$ are $G$-equivariant contractions.
Since $A\times\Z/2$ is amenable, Lemma \ref{prop:prop11} ensures that the sequence $\YY$ is equi-exact with respect to any proper length function on $A\times \Z/2$.\\ 
Moreover $D_n/{M_{(n,k)}}$ is R-dense in $\pi_n^{-1}(B_{Y_n}(e,R))$, so the uniform embeddability of $\UU(R)=(\pi_n^{-1}(B(e,R)))_n$ follows by our assumption on $(D_n/{M_{(n,k)}})_{n,k\ge n}$.
Corollary \ref{cor:cor5} can apply.
\end{proof}

\subsection{Coarse metric approximation and free product}\label{sub:sub42}
The purpose of this subsection is the conditions under which the sequence of metric groups $(D_n/{M_{(n,k)}})_{n,k\ge n}$, discussed previously, endowed with a \textit{free product metric} (cf. Theorem \ref{thm:thm3} statement) embeds uniformly into a Hilbert space.
It appears that even if the sequences $\XX=((D_n/{M_{(n,k)}},\overline{\ell}))_{n,k\ge n}$, where $\overline{\ell}$ is the quotient metric induces by a fixed length on $G$, and $\XX_{\text{Comb}}=((D_n/{M_{(n,k)}},\overline{\ell}_{\text{Comb},\BB})_{n,k\ge n}$, where $\overline{\ell}_{\text{Comb},\BB}$ is the quotient metric induced by the word metric on $\BB$ of $D$, are not coarsely equivalent (cf. remark \ref{rem:rem1}).
We can \textit{locally} approximate $\XX$ by $\XX_{Comb}$, in a particular sense, and be able to conclude on the uniform embedding of $\XX$ from the uniform embedding of $\XX_\text{Comb}$.\\

We first need an analogue of the stability by direct limit of the uniform embedding of box spaces.\\
Let $(G,\ell)$ be a discrete group endowed with a proper length $\ell$.\\
Recall that a sequence of subgroups $(H_\tau)_\tau$ of $G$ converges to $(G,\ell)$ if for all $R>0$, there exists $T>0$ such that $B(e,R)\subset H_\tau$ for all $\tau\ge T$.\\
Given a sequence of normal subgroups $\NN=(N_n)_n$ on $(G,\ell)$, we can associate to $(H_\tau)_\tau$ a sequence of metric subspaces of $X=\sqcup_{N\in \NN} G/N$ by considering $X_\tau=\sqcup_{N\in \NN} \pi_N(H_\tau)\subset X$ where $\pi_N: G\to G/N$ is the quotient map.\\
This sequence play a role analogue as in the group setting.
In particular uniform embedding is stable by direct limit through these type of sequences as the following proposition shows.

\begin{Prop}\label{prop:prop8}
Assume $(G,\ell)$ has a sequence of normal subgroups $\NN=(N_n)_n$ such that for all $\tau$, the metric space $X_\tau=\sqcup_{N\in \NN} \pi_N(H_\tau)$ uniformly embeds into Hilbert space.\\
Then the metric space $X=\sqcup_{N\in \NN} G/N$ coarsely embeds into Hilbert space.
\end{Prop}
\begin{proof}
Since $X_\tau$ embeds uniformly into Hilbert space, for all $\e>0$ we can find $\varphi_n^\tau:H_n^\tau\rightarrow \HH_n^\tau$, with $\HH_n^\tau$ a Hilbert space, such that  $\|\varphi_n(x)\|=1$ and:
\begin{description}
\item{$(i)$} for all $n\ge1$ and $x,y\in H_n^\tau$ with $d_n(x,y)\le R$, $\|\varphi_n^\tau(x)-\varphi_n^\tau(y)\|\le\e$.
\item{$(ii)$} for all $\e'>0$, there exists $M>0$ such that $|\langle \varphi_n^\tau(x)|\varphi_n^\tau(y)\rangle|\le\e'$ whenever $d_n(x,y)\ge M$.
\end{description}
with $d_n$ the left-invariant distance associated to $\ell_n$.\\
Let's fix $\sigma_n^\tau:Q_n^\tau\rightarrow G_n$, with $Q_n^\tau=G_n/{H_n^\tau}$, a section for all $n$ and $\tau$.
Let $R>0$ be a positive constant and $T$ such that $B_n(e,R)\subset H_n^\tau$ for all $n$ and $\tau\ge T$.\\
First, observe that, for all $x,y\in \Gamma_n$, $x\sim y$ if and only if $\sigma_n^\tau(\overline{x})=\sigma_n^\tau(\overline{y})$ and in such case:
$$d_n(x,y)=d_n(\sigma_n^\tau(\overline{x})x',\sigma_n^\tau(\overline{y})y')=d_n(x',y')$$
with $x'=\sigma_n^\tau(\overline{x})^{-1}x\in H_n^\tau$ (respectively, $y'=\sigma_n^\tau(\overline{y})^{-1}y\in H_n^\tau$).\\
Now, we define the induced maps $\widehat{\varphi_n^\tau}:\Gamma_n\rightarrow \HH_n^\tau\otimes\ell^2(Q_n^\tau)$ by the formula $\widehat{\varphi_n^\tau}(x)=\varphi_n^\tau(x')\otimes\delta_{\overline{x}}$ with $x'$ as above. 
A direct computation shows that these maps verify :\\
$$\|\widehat{\varphi_n^\tau}(x)-\widehat{\varphi_n^\tau}(y)\|=\left\{\begin{array}{ll}
\|\varphi_n^\tau(x')-\varphi_n^\tau(y')\|&\text{if $x\sim y$}\\
2&\text{otherwise}
\end{array}\right.$$
for $x,y\in G_n$.\\
But if $d_n(x,y)\le R$, by our assumption this imply $x^{-1}y\in B_n(e,R)\subset H_n^\tau$ and consequently $x\sim y$. 
So according to the above observation $d_n(x,y)=d_n(x',y')\le R$ and then we have: $\|\widehat{\varphi_n^\tau}(x)-\widehat{\varphi_n^\tau}(y)\|=\|\varphi_n^\tau(x')-\varphi_n^\tau(y')\|\le\e$.\\
In a similar way:
\begin{equation}\label{eq:eq2}
|\langle\widehat{\varphi_n^\tau}(x)|\widehat{\varphi_n^\tau}(y)\rangle|=\left\{\begin{array}{ll}
|\langle\varphi_n^\tau(x')|\varphi_n^\tau(y')\rangle|&\text{if $x\sim y$}\\
0&\text{otherwise}
\end{array}\right.
\end{equation}
Let $\e'>0$ be a positive constant and $M>0$ such that, $|\langle\varphi_n^\tau(x_\tau)|\varphi_n^\tau(y_\tau)\rangle|\le\e'$ whenever $d_n(x_\tau,y_\tau)\ge M$ with $x_\tau,y_\tau\in H_n^\tau$. 
According to the formula  (\ref{eq:eq2}), $|\langle\widehat{\varphi_n^\tau}(x)|\widehat{\varphi_n^\tau}(y)\rangle|\le\e'$ whenever $d_n(x,y)\ge M$.
Indeed, if $x\nsim y$ there is nothing to do, otherwise, using the observation we have that $d_n(x',y')=d_n(x,y)\ge M$ and $|\langle\widehat{\varphi_n^\tau}(x)|\widehat{\varphi_n^\tau}(y)\rangle|=|\langle\varphi_n^\tau(x')|\varphi_n^\tau(y')\rangle|\le\e'$.\\
This result combined with Lemma \ref{lem:lem1} prove the Proposition \ref{prop:prop8}.
\end{proof}

Let's get back to free product.
As usual $A$ and $B$ are two discrete finitely generated groups, $G=A\star B$ denote their free product and $D\unlhd G$ is the commutator subgroup of $G$.\\
Given two sequences of quotients of $A$ and $B$ denoted respectively $(\overline{A}_n)_n$ and $(\overline{B}_n)_n$, as before, we define a quotient sequence of $G$ by taking $G_n=\overline{A}_n\star\overline{B}_n$ and denote $\pi_n:G\to G_n$ the quotient maps associated.\\
We define $D_\tau$, with $\tau>0$, the free subgroup of $D$ with basis $\BB_\tau=\{[{a},{b}]\mid \ell({a}),\ell({b})\le\tau; a,b\neq 1\}\subset D$ (cf. Proposition \ref{prop:prop7bis}).
$H_\tau$ acts isometrically by left translation on $(\pi_n(H_\tau),\overline{\ell})$ where $\overline{\ell}$ is the quotient length associated to a proper length $\ell$ on $G$.

\begin{Lem}\label{cor:cor3} 
Let $\ell$ be a proper length function on $G$.\\
We endow  $D_\tau$ with the word metric, $\ell_{\text{Comb},\BB_\tau}$, associated to the basis $\BB_\tau$.
For all $\tau>0$, there exist a rank $N\ge0$ and a pair of constants $(C,\lambda)$ which depends only on $\tau$ such that the quotient maps $\pi_n: (H_\tau,\ell_{\text{Comb},\BB_\tau})\to (\pi_n(H_\tau),\overline{\ell})$ define $G$-equivariant $(C,\lambda)$-quasi-isometries for all $n\ge N$.
\end{Lem}
\begin{proof}
Any proper metric on a discrete countable group are equivalent, we may assume $\ell$ is a word metric associated to the union of two finite generating sets of respectively $A$ and $B$.\\
Since the sequence of quotients $(G_n)_n$ is faithful, there exists a rank $N\ge0$ such that $\BB_\tau^{(2)}=\{x.y^{-1}\mid x,y\in \BB_\tau\}$ is sent bijectively inside of $D_n$ for $n\ge N$.
Because of the freeness of the set $\pi_n(\BB)\subset G_n$ we have that $\ell_{\text{Comb}, \pi_n(\BB)}(\overline{g})=\ell_{\text{Comb},\BB_\tau}(g)$ for all $g\in D_\tau$ and $n\ge N$ where $\ell_{\text{Comb}, \pi_n(\BB)}$ is the word distance associated to $\pi_n(\BB)\subset G_n$ and $\overline{g}=\pi_n(g)$.\\
Let $T_n$ be the Bass-Serre tree of $G_n$ and $d_n$ its graph metric. 
Then $d_n(\overline{g}.\overline{A}_n,\overline{A}_n)\le \overline{\ell}(\overline{g})$ for all $\overline{g}\in G_n$ and by the inequality (\ref{eq:eq3}) of Proposition \ref{prop:prop7bis} we have:
$$\ell_{\text{Comb}, \pi_n(\BB)}(\overline{g})=\ell_{\text{Comb},\BB_\tau}(g)\le d_n(\overline{g}.\overline{A}_n, \overline{A}_n)+1\le \overline{\ell}(\overline{g})+1$$
for all $g\in D_\tau$.

Now, let $g\in D_\tau$, then ${g}$ can be written ${g}={s_{i_1}}\dots {s_{i_n}}$ with ${s_{i_k}}=[{a_{i_k}},{b_{i_k}}]\in\BB_\tau$ and $n=\ell_{\text{Comb},\BB_\tau}({g})$ in an unique way, so :
$$\overline{\ell}(\overline{g})\le\sum_{k=1}^n\ell(\overline{s_{i_k}})\le \sum_{k=1}^n2\ell(\overline{a_{i_k}})+2\ell(\overline{b_{i_k}})\le 4\tau n=4\tau\ell_{\text{Comb},\BB_\tau}(g).$$

We just proved that:
$$\ell_{\text{Comb},\BB_\tau}({g})-1\le\overline{\ell}(\overline{g})\le4\tau\ell_{\text{Comb},\BB_\tau}({g})$$
for any $g\in D_\tau$, with quasi-isometry constants which depend only on $\tau$.
\end{proof}

\begin{Rem}\label{rem:rem1}
It is interesting to note that the dilatation on all $D$, $\rho_+(t)=\sup_{\ell_{\text{Comb}}({g})\le t}\ell({g})$, depends on ${A}$ and ${B}$. Indeed, a simple computation shows that $\text{diam(${A}$)}+\text{diam(${B}$)}\le\rho_+(1)$.
This justifies what we discussed at the beginning of this subsection and our approach.\\
\end{Rem}

\begin{Cor}\label{coro5}
Let $X_\tau$ be the coarse disjoint union of the sequence of metric spaces $((\pi_n(D_\tau)/{M_{(n,k)}},\overline{\ell}))_{n,k\ge n}$ where $\overline{\ell}$ is the quotient metric induces by a fixed proper length $\ell$ on $G$.
Then $X_\tau$ is quasi-isometric to coarse disjoint union $X_{\tau,\text{Comb}}$ of the sequence $((\pi_n(D_\tau)/{M_{(n,k)}},\overline{\ell}_{\text{Comb},\BB_\tau}))_{n,k\ge n}$ where $\overline{\ell}_{\text{Comb},\BB_\tau}$ is the quotient length associated to the word distance on $\BB_\tau$.\\
In particular the uniform embedding of $X_\tau$ reduces to the uniform embedding of $X_{\tau,\text{Comb}}$.
\end{Cor}
\begin{proof}
This is a direct consequence of Lemma \ref{cor:cor3} and Proposition \ref{prop:prop5}.
\end{proof}

Since the commutator subgroups $D_n\unlhd G_n=\overline{A}_n\star\overline{B}_n$ are finite rank free groups and , for $n$ large enough, the restriction of the quotient length associated to the word distance on $\BB$ of $D$, $\overline{\ell}_{\text{Comb},\BB}$, to $D_\tau$ is nothing but the word distance on the basis $\BB_\tau$. 
The following proposition ensures the existence of a suitable sequence of finite index normal subgroups $(M_{(n,k)})_{(n,k)}$ such that the $X_{\tau,\text{Comb}}$ uniformly embed.
This concludes the proof of Theorem \ref{thm:thm1}.

\begin{Prop}\label{prop:prop10}
For all free group $\F_\oo$ of finite rank $\oo\in\N$ endowed with a word metric on a basis, there exists a sequence of finite index characteristic  subgroups of $\F_\oo$ with trivial intersection $\NN_\oo=(N_{(\oo,k)})_k$ with $N_{(\oo,k+1)}$ characteristic into $N_{(\oo,k)}$ such that the sequence $(\F_\oo/{N_{(\oo,k)}})_{\oo, k\ge\oo}$ embeds uniformly into Hilbert space.
\end{Prop}

Indeed, for all $n$, there exists $\NN_n=(N_{n,k})_k$ such that $X_{\text{Comb}}=\sqcup_{n,k\ge n}D_n/{N_{n,k}}$ embeds into Hilbert space.
Since $X_{\tau,\text{Comb}}\subset X_{\text{Comb}}$ for all $\tau>0$, the Proposition \ref{prop:prop8} concludes the proof of Theorem \ref{thm:thm1}.\\

Let us prove the Proposition \ref{prop:prop10}:
\begin{proof}
Let us recall that a characteristic subgroup $N$ of a group $G$ is a normal subgroup stable by all automorphisms of $G$.\\
Let $\F_\oo$ be a free group of finite rank $\oo\in\N$, following \cite{AGS} idea on the existence of a free group box space which coarsely embeds, we consider the sequence of characteristic subgroups  $(N_k)_k$ of $\F_\oo$,  define recursively by the formula $N_1=\langle\{w^2;w\in\F_\oo\}\rangle$, $N_{k+1}=\langle\{w^2;w\in N_{k+1}\}\rangle\unlhd N_k$. Since $N_{k+1}$ is a finite index subgroup of $N_k$ as far as the rank of $N_k$ is finite, by induction we deduce that these groups has finite index in $\F_\oo$. 
Moreover Levi theorem (cf \cite{LS}) justifies that this sequence has trivial intersection.\\
The AGS \cite{AGS} embedding result rests on a construction of a pseudo-metric $d_{W,k}$ on $X_k^\oo=\F_\oo/{N_k}$ associated to a wall structure which agrees with the combinatorial metric $d^\oo_k$ at the scale of girth($X_k^\oo$). 
Note that such a pseudo-metric always corresponds to a Hilbert metric \cite{Roe}.\\
In details, on any finite quotient $X_k^\oo$, $d_{W,k}^\oo\le d^\oo$ and :
$$d^\oo(x,y)\le\text{girth($X_k^\oo$)} \text{ if and only if } d_{W,k}^\oo(x,y)\le\text{girth($X_k^\oo$)}$$
for any $x,y\in X_k^\oo$. Furthermore, if one of the above inequalities hold, then $d^\oo(x,y)=d_{W,k}^\oo(x,y)$ \cite{AGS}.

So to prove our proposition a sufficient condition is to prove that girth($X_k^\oo$) grows uniformly on the rank $\oo$, i.e there exists a proper function $f:\N\rightarrow\N$ such that $\text{girth($X_k^\oo$)}\ge f(k)$ for all $\oo\ge 1$ and $k\ge\oo$. 
Indeed, take in account the relations between $d_W^\oo$ and $d^\oo$ described above, we just need to prove the existence of a proper function $\rho_-:\mathbb{R}_+\rightarrow\mathbb{R}_+$ such that:
\begin{equation}
\rho_-(d_\oo(x,y))\le d_W^\oo(x,y)
\end{equation}
for any $\oo$, $k\ge\oo$ and $x, y\in X_k^\oo$.\\
We must prove that $\rho(t)=\inf_{\{(x,y);d_\oo(x,y)\ge t\}}d_{W,k}^\oo(x,y)$ is proper, i.e for any $R>0$, there exists $R'>0$ such that for any $\oo$, $d_{W,k}^\oo(x,y)\le R$ implies $d_\oo(x,y)\le R'$.\\
If such a function $f$ exists. 
Let $K_0$ such that $f(k)>R$ for any $k\ge K_0$ and $R'\ge \max_{\oo\le K_0}\text{diam($\sqcup_{K_0\ge k\ge\oo}X_k^\oo$)}+R$.
We have that for any $x,y\in X_k^\oo$ with $k\ge\oo$, if $k\le K_0$ by construction $d(x,y)\le \max_{\oo\le K_0}\text{diam($\sqcup_{K_0\ge k\ge\oo}X_k^\oo$)}\le R'$ otherwise $k\ge K_0$ and $\text{girth($X_k^\oo$)}\ge f(k)\ge R$ so $d_W(x,y)=d(x,y)\le R\le R'$ as required.

Now the existence of $f$ follows from the fact that the sequence $(N_{(\oo,k)})_k$ is characteristic and $N_{(\oo,k+1)}$ contains no basis element of $N_{(\oo,k)}$.
Indeed, automorphisms act transitively on basis, so by contradiction this will implies that $N_{(\oo,k+1)}=N_{(\oo,k)}$ which is absurd because the sequence is defined recursively and has trivial intersection. 
Proposition 3.2 \cite{LS} ensures that relatively to any basis of $\F_\oo$, $\ell_{\text{Comb},\oo}(g)\ge k$ for any $g\in N_k^\oo$ and so $\text{Girth($X_k^\oo$)}\ge k$.
\end{proof}

\section{Further on finite quotients of free products}\label{sec:sec5}
In this section the groups $A$ and $B$ we consider are supposed to be finitely generated.\\
We start this section by presenting another construction of finite quotients of a free product due to Baumslag \cite{BT} that is extensively use here.\\
First, assume $A$ and $B$ are two finite groups and $G=A\star B$ their free product.
Let $U(k)\subset G$ be the set of words with normal form: $g_l\dots g_1$ with $l\le k$.
For $k\ge 2$, $U(k)$ decomposes as follows:
$$U(k)=U(k-1)\cup (V(k)\cap A.U(k-1))\cup(V(k)\cap B.U(k-1))$$
with $V(k)=U(k)\setminus U(k-1)$ the set of words with normal form: $g_k\dots g_1$.\\
The action by left multiplication of $A$ on $G$ stabilizes $U(k-1)\cup (V(k)\cap A.U(k-1))$ and can be extend to all $U(k)$ by assuming that $A$ acts trivially on $V(k)\cap B.U(k-1)$.
Something similar can be achieved with $B$ by replacing $V(k)\cap A.U(k-1)$ by $V(k)\cap B.U(k-1)$.\\
This gives two symmetric representations $\sigma_A(k): A\rightarrow Sym(U(k))$ and $\sigma_B(k): B\rightarrow Sym(U(k))$ and by the universal property of free product we obtain: $\sigma(k)=\sigma_A(k)\star\sigma_B(k) :G\rightarrow Sym(U(k))$.
Because $A$ and $B$ are finite $U(k)$ is also finite, in particular $N_k=\ker(\sigma(k))$ has finite index in $G$.

\begin{Rem}
The sequence of finite index normal subgroups $(N_k)_k$ has trivial intersection.
Indeed, let $g=g_{n_0}\dots g_1$ be the reduced form of $g\in G$ with $g\neq e$. Then $\sigma(k)(g)(e)=g\neq e$ for any $n\ge n_0$.
\end{Rem}

Now assume $A$ and $B$ are both residually finite groups. Let $(\overline{A}_n)_n$ and $(\overline{B}_n)_n$ be two faithful sequence of finite quotients of respectively $A$ and $B$. 
Let's denote $G_n=\overline{A}_n\star\overline{B}_n$ and $\sigma_n(k): G_n\rightarrow Sym(U_n(k))$ the representation built previously applied to the free product of finite groups $G_n$. Then the subgroups $N_{(n,k)}=\ker(\sigma_n(k)\circ\pi_n)$ of $G$ where $\pi_n:G\rightarrow G_n$ is the quotient map, form a sequence of finite index normal subgroups with trivial intersection.

\subsection{Free product box spaces and expansion property}
Here we study expansion property of the sequence $(G/{N_{(n,k)}})_{n,k\ge n}$.\\

Let $X=(V,E)$ be a finite graph and $A\subset V$ a subset of vertices.
The \textbf{edge boundary} $\partial A\subset E$ of $A$ is the set of edges which relate an element $A$ to one of $V\setminus A$.
\begin{Def}\label{def:def12}
Given a finite graph $X=(V,E)$, the \textbf{Cheeger constant} $h(X)$ is defined as $\min_{A\subset V\mid |A|\le\frac{|V|}{2}}\frac{|\partial A|}{|A|}$.
\end{Def}
A sequence of finite graphs $(X_n)_n$ is called \textbf{expander} if there exists $c>0$ such that $h(X_n)\ge c$ for any $n$.\\

Let $G$ be a group which acts on a set $Z$ and $S$ a generating set of $G$.
We call \textbf{Schreir graph} and denote it $X_S(Z)$ the graph defined by: $V(X_S(Z))=Z$ and $(z,z')\in E(X_S(Z))$ if $z'=s.z$ with $s\in S$.\\
Given a sequence of finite quotients $(Q_n)_n$ of $G$, we construct a sequence of Schreir graphs $X_S(Q_n)$.
Expansion of $(X_S(Q_n))_n$ does not depends on $S$ and we say that $(Q_n)_n$ is expander if such a $S$ exists.

Proposition \ref{prop:prop3} is a consequence of the following:
\begin{Prop}\label{prop:prop14} 
Let $A$ and $B$ be two residually finite groups and $G=A\star B$ their free product.
The faithful sequence of finite quotients $(G_n/{N_{n,k}})_{(n,k\ge n)}$ of $G$ is non-expander.
\end{Prop}
In particular we cannot expect a free product with property $(\tau)$ (cf \cite{Lu}).\\

In order to prove Proposition \label{prop:prop14}  we need a intermediate result.
\begin{Lem}\label{lem:lem6}
Let $G$ be a finite group, $S$ a symmetric generating set of $G$ and $H\le G$ a subgroup of $G$.
Then $h(G,S)\le h(G/H,\overline{S})$.
\end{Lem}
\begin{proof}
Let $X\subset G$ be a subset. 
Then $|X.H|=|X|.|H|$ if and only if the multiplication map 
$X\times H \rightarrow G$ is bijective and this is equivalent to say that $(X^{-1}.X)\cap H=\{e\}$.\\
Now let $Y\subset G/H$, there exists $X\subset G$ which is sent bijectively to $Y$ by the projection $\pi:G\rightarrow G/H$ such that $X^{-1}.X\cap H=\{e\}$.
Indeed, let's define ${Y}^{(2)}=\{\overline{a^{-1}a'}\mid a,a'\in \tilde{Y}\}$ with $\tilde{Y}\subset G$ a arbitrary set of representative of $Y$, there exists $X_0\subset G$ which is sent bijectively to ${Y}^{(2)}$. In particular $X_0.H\cap X_0=\{e\}$ and the representative of $Y\subset {A'}^{(2)}$ in $X_0$ do the job.\\
It's enough to prove that for any $Y\subset G/H$ with $|Y|\le\frac{|G/H|}{2}$ and a set of representative $X\subset G$ as above:
\begin{equation}\label{eq:eq6}
\frac{|\partial X.H|}{|X.H|}\le\frac{|\partial Y|}{|Y|}.
\end{equation}
Indeed, $|X.H|=|X|.|H|\le\frac{|G/H|}{2}.|H|=\frac{|G|}{2}$.\\
Observe that for a edge $e=\{xh,sxh\}\in E(X.H, G\setminus (X.H))$ between a element of $X.H$ and $G\setminus (X.H)$, the edges $e.h'=\{xhh',sxhh'\}$ with $h'\in H$ are distincts and still belong to $E(X.H, G\setminus (X.H))$. Furthermore they represent the fiber of the edge 
$\overline{e}=\{\overline{x},s.\overline{x}\}\in E(Y, (G/H)\setminus Y)$.
It follows that $|\partial X.H|=|\partial Y|.|H|$. 
Finally the equality $|X.H|=|X|.|H|$ implies inequality \ref{eq:eq6}.
\end{proof}

Lemma \ref{lem:lem6} reduces Proposition \label{prop:prop14} to the study of the expansion property of Schreir graphs $(X_S(U_n(k)))_{n,k\ge n}$.
Indeed, $U_n(k)\simeq {G/N_{(n,k)}}/St(e)$ where $St(e)$ is the stabilizer of $e\in U_n(k)$ .

\begin{Prop}
The sequence of finite homogenous graphs $(X_S(U_n(k)))_{n,k\ge n}$ is non-expander.
\end{Prop}
\begin{proof}
We take the Schreir graph structure associated to the finite generating $S=S_A\cup S_B$ of $G$ with $S_A$ and $S_B$ two finite generating sets of respectively $A$ and $B$.\\
Let's consider $P_n^A(k)\subset U_n(k)$ (respectively, $P_n^B(k)\subset U_n(k)$) the set of words with reduced form $g_k\dots g_1$ with $k\le n$ and $g_1\in \overline{A}_n$ (respectively, $g_1\in \overline{B}_n$).
Given $w\in U_n(k)$ with $w\neq e$, $w$ belongs to $P_n^A(k)$ or $P_n^B(k)$.
Assume $w\in U_{n,k}^A$, if we denote its set of neighbors by $V(w)=\{s.w\mid s\in S\}$, then $V(w)\subset U_{n,k}^A$.
In particular $\partial P_n^A(k)\subset V(e)$ and it follows that $|\partial P_n^A(k)|\le |S|$.\\
Furthermore we have 
$|P_n^A(k)|\ge (|\overline{A}_n|-1)^{\lfloor n/2 \rfloor}(|\overline{B}_n|-1)^{\lfloor n/2\rfloor}$ which implies:
$$\frac{|\partial P_n^A(k)|}{|P_n^A(k)|}\le\frac{|S|}{(|\overline{A}_n|-1)^{\lfloor n/2 \rfloor}(|\overline{B}_n|-1)^{\lfloor n/2\rfloor}}\rightarrow 0$$
when $n$ or $k$ goes to infinity and similar holds for $P_n^B(k)$.
Since the sets $P_n^A(k)$ and $P_n^B(k)$ verify $|P_n^A(k)|+|P_n^B(k)|=|U_n(k)|+1$, one of them has cardinality lower or equal then $\frac{|U_n(k)|}{2}$, so $h(U_n(k))\rightarrow 0$ whenever $k$ or $n$ go to infinity.
\end{proof}

\subsection{Free product symmetric representations and uniform embedding.}
Let $G$ be a group, we call \textbf{irreducible symmetric representation} of $G$ an action of type: $G\curvearrowright G/N$ with $N\le G$ a finite index subgroup of $G$.
This denomination comes from the fact that these actions give symmetric representations: $G\rightarrow \text{Sym($G/N$)}$ and any transitive action on a finite set are obtained this way.\\
A sequence of irreducible symmetric representations will be for us a sequence of finite index subgroups, non-necessarily normal, in $G$.
Moreover, we say that it is faithful if these subgroups has trivial intersection.\\
Given such a representation $N\le G$ and a generating set $S$ of $G$. 
We denote the Schreir graph $X_S(G/N)$ associated to $N$ and $S$ simply by $X_S(N)$.\\

\noindent\textbf{Proposition 1.} \textit{Let $A$ and $B$ be two finitely generated residually finite groups, $G=A\star B$ their free product and $S\subset G$ a finite generating set of $G$.\\
Assume $A$ and $B$ admit respectively an embeddable box space. Then there exists a faithful sequence of irreducible symmetric representation $(N_n)_n$ of $G$ such that $\sqcup_n(X_S(N_n))_n$ embeds into a Hilbert space.}\\

By using irreducible symmetric representation we relaxed the assumption that the subgroups we consider are normal but these notion still keep track of certain properties as properties $(\tau)$ \cite{Lu} and (A). 
More precisely we can still prove that property (T) (respectively, amenability) implies property $(\tau)$ (respectively, property (A)).\\
In a other hand we lose on the approximation properties correspondance:\\

\noindent\textbf{Proposition 2.} \textit{There exists a faithful sequence of irreducible symmetric representation $(N_n)_n$ of $\F_2$, such that $\sqcup_n \F_2/{N_n}$ has property (A).}\\

These two Proposition's are actually consequences of the following theorem:
Let $A$ and $B$ be two discrete finitely generated residually finite groups.
Let's denote $(\overline{A}_n)_n$ and $(\overline{B}_n)_n$ two faithful sequences of finite quotients of respectively $A$ and $B$ and $G_n=\overline{A}_n\star \overline{B}_n$ their free products.
\begin{Thm}\label{thm:thm6}
If $\sqcup_n\overline{A}_n$ and $\sqcup_n\overline{B}_n$ embed into Hilbert space (respectively, has property (A)), then $\sqcup_n G_n$ embeds into a Hilbert space (respectively, has property (A)).
\end{Thm}

Let $S=S_A\cup S_B\subset G$ be a finite generating set of $G$ built by the union of a generating set of $A$ and $B$.\\
By considering Baumslag construction, then $U_n(k)\simeq G/{St(e)}$, with the Schreir graph structure induced by $S$, embeds isometrically in $(G_n,\overline{S})$ and is our needed sequence of symmetric representation.
\begin{Prop}\label{prop:prop18}
Let $\TT=\{(T,d_T)\}_T$ be the family of bounded geometry trees endowed with their graph metric. Then $\TT$ forms an equi-exact family of metric spaces.
\end{Prop}
\begin{proof}
Let's fix $x_T\in T$ on any $T\in\TT$ and denote $(x|y)_T=\frac{1}{2}[|x|_T+|y|_T-d(x,y)]$ with $|x|_T=d_T(x,x_T)$ with $x,y\in T$.\\
Recall that trees are $0$-hyperbolic space and so satisfy the following hyperbolic condition:
$$(x|z)_T\ge \min \{(x|y)_T,(y|z)_T\}$$
for any $x,y,z\in T$.
In particular given any positive constant $L\ge0$, the relation $x\sim y$ if $(x|y)_T\ge L$ is an equivalence relation on $T_{\ge L}=\{x\in T\mid|x|_T\ge L\}$.\\
Let $L>0$ be a positive constant and $A_{T,L,k}=\{x\in T\mid kL\le|x|_T\le k(L+1)\}$ an annulus of $T_{\ge L}$.
The relation on $A_{T,L,k}\subset T_{\ge L}$: $x\sim y$ if $(x|y)_T\ge L(k-1/2)$ is an equivalence relation on $A_{T,L,k}$. Moreover, $x\sim y$ implies $d_T(x,y)\le 3L$ and $x\nsim y$ implies $d_T(x,y)\ge L$.
This induces a partition $A_{T,L,k}=\cup_{i\in I_{T,L,k}} C_{T,L,k,i}$ such that $d_T(C_{T,L,k,i},C_{T,K,k',i})\ge L$ whenever $k\neq k'$ and $\text{diam($C_{T,K,k,i}$)}\le 3L$.\\
Using that decomposition we construct the covering $\UU=\UU_1\cup\UU_2$ with $\UU_1=\{C_{T,L,2k,i}\mid k\ge0,i\in I_{T,L,2k}\}$ and $\UU_2=\{C_{T,L,2k+1,i}\mid k\ge0,i\in I_{T,L,2k+1}\}$ of $T$ with the property that $\UU_i$ forms a L-separated, i.e for any $C,C'\in \UU_i$, $d(C,C')\ge L$, 3L-uniformly bounded family with $i=1,2$.\\
Let $\WW=\WW_1\cup\WW_2$ with $\WW_i=\UU_i(L/2)=\{C(L/2)\mid C\in\UU_i\}$ which is the covering obtained by taking the $L/2$-neighborhood of each element of $\UU$.
Since at most two elements of $\UU$ intersect a ball of radius $L/2$, at most two elements of $\WW$ contain a given point. Moreover any ball of radius $L/2$ are contained in an element of $\WW$. 
This fact will be use later for estimation.\\
Let's define $\phi_{T,L}: T\rightarrow \ell^1(\cup_{k\in\N} I_{T,L,k})$ by the formula 
$$\phi_{T,L}(x)(i)=\frac{d_T(x,T\setminus W_{T,L,k,i})}{\sum_{k'\in \N, j\in I_{T,L,k}}d_T(x,T\setminus W_{T,L,k',j})}$$
for any $x\in T$ and $i\in I_{T,L,k}$.\\
Then for $k\in\N$ and $i\in I_{T,L,k}$, $\phi_{T,L}(.)(i)$ is supported in $W_{T,L,k,i}$ which has diameter lower or equal to $4L$ and so forms a uniformly bounded supported partition of the unity. 
Moreover: 
\begin{gather*}
|\phi_{T,L}(x)(i)-\phi_{T,L}(y)(i)|\le \frac{d(x,y)}{\sum_{k'\in \N, j\in I_{T,L,k}}d_T(x,T\setminus W_{T,L,k',j})}\\
+d_T(y,T\setminus W_{T,L,k,i})\bigg|\frac{1}{\sum_{k'\in \N, j\in I_{T,L,k}}d_T(x,T\setminus W_{T,L,k',j})}-\frac{1}{\sum_{k'\in \N, j\in I_{T,L,k}}d_T(y,T\setminus W_{T,L,k',j})}\bigg|\\
\le \frac{2d(x,y)}{L}+\frac{2}{L}\bigg(\sum_{k'\in \N, j\in I_{T,L,k}}|d_T(x,T\setminus W_{T,L,k',j})-d_T(y,T\setminus W_{T,L,k',j})|\bigg)\\
\le \frac{10}{L}d(x,y)
\end{gather*}
It follows that:
\begin{gather*}
\|\phi_{T,L}(x)-\phi_{T,L}(y)\|_1\le \sum_{k\in \N, i\in I_{T,L,k}}|\phi_{T,L}(x)(i)-\phi_{T,L}(y)(i)|\\
\le \frac{40}{L}d(x,y).
\end{gather*}
This does not depend on the tree $T$ we consider. By taking $L>0$ large enough we can check the Definition \ref{def:def10}.
\end{proof}

\begin{Prop}\label{prop:prop19}
Let $(X_n)_n$ be a sequence of metric spaces with an uniformly embeddable (respectively, equi-exact) cover $\UU=\{\UU_n\}_n$.
 Assume for any $L>0$ and $n\ge 1$, there exists $Z_n\subset X_n$ a subspace of $X_n$ such that $(Z_n)_n$ is an embeddable sequence and $\{U\setminus Z_n\mid U\in \UU_n\}$ is L-separated in $X_n$ for any $n$.
Then $(X_n)_n$ is uniformly embeddable (respectively, equi-exact).
\end{Prop}
\begin{proof}
Let $L>0$ be a positive constant and $Z_n\subset X_n$ as in the statement. Then $\WW=\WW_1\cup\WW_2$ with $\WW_1=\{Z_n(L/2)\}$ and $\WW_2=\{(U\setminus Z_n)(L/2)\mid U\in \UU_n\}$ is a cover such that at most two elements of $\WW$ contain a same element of $X_n$ and any ball of radius $L/2$ is contained in a element of $\WW$.
As in the proof of Proposition \ref{prop:prop18} we can build a $O(L^{-1})$-lipschitz partition of unity subordinated to $\WW$ and apply Proposition \ref{prop:prop13} (respectively, Proposition \ref{prop:prop13bis}) to conclude.
\end{proof}

We present a brief introduction to tree of metric spaces formalism (cf. \cite{DMG}) that we use to prove Theorem \ref{thm:thm6}.\\
Let $T=(V,E)$ be an oriented (this is for relieve on notations) connected tree.\\
A tree of metric spaces on $T$ is given by a family of metric spaces $\XX=\{(X_v,d_v)\}_{v\in V}$ and a map $f:E\rightarrow \cup_{v,v'\mid(v,v')\in E}X_v\times X_{v'}$ such that $f(e)\in X_{s(e)}\times X_{t(e)}$.\\
The \textbf{total space} $X$ associated to $(\XX,f)$ is a metric space defined by $X=\sqcup_{v\in V} X_v$ and a metric $d_X$ obtained as follows:\\
Let's consider the subset $\DD=(\cup_{v\in V} X_v\times X_v)\cup f(E)$ of $X\times X$. 
Then $\DD$ defines a ample domain, i.e for any pair of elements $x,y\in X$, there exists a path $x_1=x,\dots,x_n=y$ such that $(x_i,x_{i+1})\in\DD$ for any $1\le i\le n-1$. 
The metric $d_X$ is defined as the unique enveloping metric of the partial metric $\widehat{d}:\DD\rightarrow \R_+$ given by the formulas $\widehat{d}|_{X_v\times X_v}=d_v$ and $\widehat{d}\circ f(e)=1$ for any $e\in E$ (cf \cite{DMG}).\\
Now, given a free product $G=A\star B$ endowed with a proper left invariant metric $d$ and its Bass-Serre tree $T=(G/A\cup G/B, G)$.
We associate a tree of metric spaces by taking $X_{gA}=(gA,d|_{gA})$ and $X_{gB}=(gB,d|_{gB})$ and $f: G\rightarrow G/A\times G/B$ with $f(g)=(gA,gB)$.\\
$G$ acts on the total spaces $X_{(G,d)}$ by the formula $g'.(ga)=gg'a\in X_{g'gA}$ with $g'\in G$ and $ga\in X_{gA}$ and similarly on $G/B$.
Moreover this action is isometric, free and has two orbits.
The formula for the distance on $X_{(G,d)}$ given in Proposition (5.5) of \cite{DMG} implies the existence of a pair of positive constants $(C,\lambda)$ such that any free product $(G,d)$ of finitely generated groups  endowed with a proper left invariant metric $d$ is $(C,\lambda)$-quasi-isometric to $X_{(G,d)}$.
In other words it is equivalent to prove Theorem \ref{thm:thm3} for $(G_n)_n$ or $(X_{(G_n,d_n)})_n$ (cf. Proposition \ref{prop:prop5}).

We can prove Theorem \ref{thm:thm6}.
\begin{proof}
Let $G_n$ as in the statement, $T_n$ its Bass-Serre tree and $X_n$ its tree of metric spaces.
The sequence $(G_n)_n$ are uniformly quasi-isometric to $\XX=(X_n)_n$, so the proof reduce to show that $\XX$ embeds uniformly (respectively, has (A)) if $(\overline{A}_n)_n$ and $(\overline{B}_n)_n$ embed (respectively, has (A)).\\
The actions of $G_n$ on $T_n$ are uniformly cobounded and Proposition \ref{prop:prop18} ensures that $\TT=(T_n)_n$ forms an equi-exact sequence.
Moreover $p_n: X_n\rightarrow T_n$ is a sequence of $G_n$-equivariant contractions (this is justified by Proposition 5.5 of \cite{DMG}).
To apply Corollary \ref{cor:cor5}, we must prove that $p_n^{-1}(B(\overline{A}_n,r))\subset X_n$ (here $\overline{A}_n\in V(T_n)$) embeds uniformly into a Hilbert space (respectively, has (A)) for any $r\ge 0$.\\
Let's do this by induction on $r$.
For $r=0$, $p_n^{-1}(\overline{A}_n)=X_{\overline{A}_n}$ which is isometric to $\overline{A}_n$ for any $n$ and is an embeddable sequence (respectively, equi-exact sequence) by assumption.\\
Assume now $r\ge 1$, if we denote $C_n(r)=p_n^{-1}(B(\overline{A}_n,r))$, we have the following decomposition:
$$C_n(r)=C_n(r-1)\cup\bigcup_{v\in T_n\mid d_{T_n}(\overline{A}_n,v)=r} X_v.$$
Let $L>0$ be positive constant and $Y_v\subset X_v$ be the subspace of $X_v$ forms by elements at distance at most $L/2$ from $C_n(r-1)$ for $v\in T_n$.
Then $C_n'(r-1)=C_n(r-1)\cup\bigcup_{v\in T_n\mid d_{T_n}(\overline{A}_n,v)=r} Y_v$ is quasi-isometric to $C_n(r-1)$ and by induction hypothesis forms an embeddable (respectively, equi-exact) sequence.\\
In the other hand $C_n(r)\setminus C_n'(r-1)=\bigcup_{v\in T_n\mid d_{T_n}(A_n,v)=r} X_v\setminus Y_v$ is L-separated in $C_n(r)$ and Proposition \ref{prop:prop19} implies that $(C_n(r))_n$ embeds uniformly (respectively, is equi-exact).
We can apply Corollary \ref{cor:cor5} to conclude.
\end{proof}

\bibliographystyle{plain}
\bibliography{biblioo}

\begin{thebibliography}{10}

\bibitem{Arn}
S.~Arnt.
\newblock Fibred coarse embeddability of box spaces and proper isometric affine
  actions on {$L^p$} spaces.
\newblock {\em Bull. Belg. Math. Soc. Simon Stevin}, 23(1):21--32, 2016.

\bibitem{AGS}
Goulnara Arzhantseva, Erik Guentner, and J{\'a}n Spakula.
\newblock Coarse non-amenability and coarse embeddings.
\newblock {\em Geom. Funct. Anal.}, 22(1):22--36, 2012.

\bibitem{AzTe}
Goulnara Arzhantseva and Romain Tessera.
\newblock Admitting a coarse embedding is not preserved under group extensions.
\newblock 05 2016.

\bibitem{BT}
Benjamin Baumslag and Marvin Tretkoff.
\newblock Residually finite {HNN} extensions.
\newblock {\em Comm. Algebra}, 6(2):179--194, 1978.

\bibitem{BHV}
Bachir Bekka, Pierre de~la Harpe, and Alain Valette.
\newblock {\em Kazhdan's property ({T})}, volume~11 of {\em New Mathematical
  Monographs}.
\newblock Cambridge University Press, Cambridge, 2008.

\bibitem{BrOz}
Nathanial~P. Brown and Narutaka Ozawa.
\newblock {\em {$C^*$}-algebras and finite-dimensional approximations},
  volume~88 of {\em Graduate Studies in Mathematics}.
\newblock American Mathematical Society, Providence, RI, 2008.

\bibitem{CWW}
Xiaoman Chen, Qin Wang, and Xianjin Wang.
\newblock Characterization of the {H}aagerup property by fibred coarse
  embedding into {H}ilbert space.
\newblock {\em Bull. Lond. Math. Soc.}, 45(5):1091--1099, 2013.

\bibitem{CCJJV}
Pierre-Alain Cherix, Michael Cowling, Paul Jolissaint, Pierre Julg, and Alain
  Valette.
\newblock {\em Groups with the {H}aagerup property}, volume 197 of {\em
  Progress in Mathematics}.
\newblock Birkh{\"a}user Verlag, Basel, 2001.
\newblock Gromov's a-T-menability.

\bibitem{DMG}
Marius Dadarlat and Erik Guentner.
\newblock Constructions preserving {H}ilbert space uniform embeddability of
  discrete groups.
\newblock {\em Trans. Amer. Math. Soc.}, 355(8):3253--3275, 2003.

\bibitem{DG}
Marius Dadarlat and Erik Guentner.
\newblock Uniform embeddability of relatively hyperbolic groups.
\newblock {\em J. Reine Angew. Math.}, 612:1--15, 2007.

\bibitem{Fa}
Benson Farb and Dan Margalit.
\newblock {\em A primer on mapping class groups}, volume~49 of {\em Princeton
  Mathematical Series}.
\newblock Princeton University Press, Princeton, NJ, 2012.

\bibitem{G}
M.~Gromov.
\newblock Asymptotic invariants of infinite groups.
\newblock In {\em Geometric group theory, {V}ol.\ 2 ({S}ussex, 1991)}, volume
  182 of {\em London Math. Soc. Lecture Note Ser.}, pages 1--295. Cambridge
  Univ. Press, Cambridge, 1993.

\bibitem{Lu}
Alex Lubotzky.
\newblock What is{$\dots$}property {$(\tau)$}?
\newblock {\em Notices Amer. Math. Soc.}, 52(6):626--627, 2005.

\bibitem{LS}
Roger~C. Lyndon and Paul~E. Schupp.
\newblock {\em Combinatorial group theory}.
\newblock Classics in Mathematics. Springer-Verlag, Berlin, 2001.
\newblock Reprint of the 1977 edition.

\bibitem{Mar}
G.~A. Margulis.
\newblock Explicit constructions of expanders.
\newblock {\em Problemy Pereda\v ci Informacii}, 9(4):71--80, 1973.

\bibitem{Roe}
John Roe.
\newblock {\em Lectures on coarse geometry}, volume~31 of {\em University
  Lecture Series}.
\newblock American Mathematical Society, Providence, RI, 2003.

\bibitem{Roe2}
John Roe.
\newblock Hyperbolic groups have finite asymptotic dimension.
\newblock {\em Proc. Amer. Math. Soc.}, 133(9):2489--2490, 2005.

\bibitem{Ser}
Jean-Pierre Serre.
\newblock {\em Trees}.
\newblock Springer-Verlag, Berlin-New York, 1980.
\newblock Translated from the French by John Stillwell.

\bibitem{STY}
G.~Skandalis, J.~L. Tu, and G.~Yu.
\newblock The coarse {B}aum-{C}onnes conjecture and groupoids.
\newblock {\em Topology}, 41(4):807--834, 2002.

\bibitem{WY}
Rufus Willett and Guoliang Yu.
\newblock Higher index theory for certain expanders and {G}romov monster
  groups, {I}.
\newblock {\em Adv. Math.}, 229(3):1380--1416, 2012.

\bibitem{Yu}
Guoliang Yu.
\newblock The coarse {B}aum-{C}onnes conjecture for spaces which admit a
  uniform embedding into {H}ilbert space.
\newblock {\em Invent. Math.}, 139(1):201--240, 2000.

\end{thebibliography}
\end{document}